\documentclass[a4paper,11pt]{article}
\usepackage{amsmath,amssymb,amsthm}
\usepackage{enumitem,bm}
\usepackage{color}
\usepackage[margin=30mm]{geometry}

\newcommand{\N}{\mathbb{N}}
\newcommand{\R}{\mathbb{R}}

\newcommand{\E}{\mathcal{E}}
\newcommand{\LL}{\mathcal{L}}
\newcommand{\M}{\mathcal{M}}
\newcommand{\J}{\mathcal{J}}
\newcommand{\dH}{\, d\Gamma}
\newcommand{\der}{\mathrm{D}}

\newcommand{\dx}{\, dx}
\newcommand{\dt}{\, dt}

\newcommand{\ds}{\, ds}
\newcommand{\dz}{\, dz}
\newcommand{\pd}{\partial}
\newcommand{\pdnu}{\pd_{\bm{\nu}}}
\newcommand{\bH}{\bm{H}}
\newcommand{\bV}{\bm{V}}
\newcommand{\bW}{\bm{W}}
\newcommand{\bZ}{\bm{Z}}

\newcommand{\abs}[1]{\left| #1 \right|}
\newcommand{\norm}[1]{\| #1 \|}

\newcommand{\inner}[2]{\langle #1 , #2 \rangle}

\newcommand{\eps}{\varepsilon}
\newcommand{\Lap}{\Delta}
\newcommand{\surf}{\nabla_{\Gamma}}
\newcommand{\LB}{\Delta_{\Gamma}}

\theoremstyle{plain}
\newtheorem{thm}{Theorem}[section]
\newtheorem{prop}[thm]{Proposition}
\newtheorem{lem}[thm]{Lemma}

\theoremstyle{plain}

\numberwithin{equation}{section}

\allowdisplaybreaks[4]

\begin{document}

\title{Convergence to equilibrium for a bulk--surface Allen--Cahn system coupled through a Robin boundary condition}

\author{Kei Fong Lam\footnotemark[1] \and Hao Wu\footnotemark[2]}

\date{ }

\maketitle

\renewcommand{\thefootnote}{\fnsymbol{footnote}}
\footnotetext[1]{Department of Mathematics, The Chinese University of Hong Kong, Shatin, N.T., Hong Kong ({\tt kflam@math.cuhk.edu.hk}).}
\footnotetext[2]{School of Mathematical Sciences and Shanghai Key Laboratory for Contemporary Applied Mathematics, Fudan University, Shanghai 200433, China;
Key Laboratory of Mathematics for Nonlinear Sciences (Fudan University), Ministry of Education, Shanghai 200433, China
({\tt haowufd@fudan.edu.cn}).}

\begin{abstract}
We consider a coupled bulk--surface Allen--Cahn system affixed with a Robin-type boundary condition between the bulk and surface variables. This system can also be viewed as a relaxation to a bulk--surface Allen--Cahn system with non-trivial transmission conditions. Assuming that the nonlinearities are real analytic, we prove the convergence of every global strong solution to a single equilibrium as time tends to infinity. Furthermore, we obtain an estimate on the rate of convergence. The proof relies on an extended {\L}ojasiewicz--Simon type inequality for the bulk--surface coupled system. Compared with previous works, new difficulties arise as in our system the surface variable is no longer the trace of the bulk variable, but now they are coupled through a nonlinear Robin boundary condition.
\end{abstract}

\noindent \textbf{Key words. } Allen--Cahn equation, bulk--surface interaction, global well-posedness, convergence to equilibrium, {\L}ojasiewicz--Simon inequality \\

\noindent \textbf{AMS subject classification. } 35B40, 35D35, 35K20, 35K61

\section{Introduction}
The coupling of bulk and surface (boundary) interactions can be found in various physical processes with boundary effects, for instance, the phase separation of binary mixtures with effective short-range interaction between the mixture and the solid wall \cite{K-etal}, the moving contact line problem \cite{Q1}, the heat conduction problem with a boundary heat source \cite{FGGR},
certain Markov process with diffusion, absorption and reflection on the boundary \cite{Ta}, and so on.
From the mathematical point of view, those nontrivial dynamics on the boundary that serve to account for the influences of the boundary to the bulk dynamics inside the domain are described by the so-call dynamic boundary condition, as besides the spatial derivatives it also involves temporal derivatives of the unknown variable (in some specific cases, its variant is also referred to as the Wentzell boundary condition). As an illustrating example, we recall a generic heat equation with dynamic boundary condition posed in a bounded domain $\Omega$ with boundary $\Gamma := \pd \Omega$ that reads as
\begin{align*}
\pd_t u = \Lap u + g_b\ \  \text{ in } \Omega, \quad \gamma \pd_t u = \sigma \LB u - \pdnu u - \kappa u + g_s\ \  \text{ on } \Gamma,
\end{align*}
where $\gamma$, $\sigma$, $\kappa$ denote non-negative constants, $g_b$ and $g_s$ are prescribed external heat sources located in the domain and on the boundary, $\pdnu u := \nabla u \cdot \nu$ is the normal derivative of $u$ on $\Gamma$ with unit outward normal $\nu$, and $\LB$ denotes the Laplace--Beltrami operator on $\Gamma$.
The above problem can be rewritten as the following bulk--surface system
\begin{align*}
&\pd_t u = \Lap u + g_b\ \  \text{ in } \Omega, \quad \text{with}\quad  u \vert_\Gamma = u_\Gamma \ \  \text{ on } \Gamma,\\
&\gamma \pd_t u_\Gamma = \sigma \LB u_\Gamma - \pdnu u - \kappa u_\Gamma + g_s\ \  \text{ on } \Gamma,
\end{align*}
 for the bulk variable $u$ defined in $\Omega$ as well as a surface variable $u_\Gamma$ defined on $\Gamma$. In particular, the above nonhomogeneous Dirichlet boundary condition for $u$ on $\Gamma$ turns out to be a transmission condition that connects the two variables by imposing that $u_\Gamma$ is the trace of $u$ on $\Gamma$.
 In this paper, we focus our interests on the coupling of bulk and surface interactions in certain phase separation process.
One typical model is the Allen--Cahn equation \cite{AC}, a second order semilinear parabolic equation, whose connection with the motion by mean curvature for free interfaces has been well-established by several authors \cite{Chen,DS,I}. Together with its fourth-order counterpart i.e., the Cahn--Hilliard equation \cite{CH}, these so-called phase field models have become essential components of more complex mathematical models for the evolution of multi-phase phenomena.

The bulk--surface coupled Allen--Cahn system that we are going to investigate reads as follows
\begin{subequations}\label{ACAC}
\begin{alignat}{2}
\pd_t u - \Lap u + f(u) = 0 &\quad  \text{ in } Q := \Omega \times (0,\infty), \label{u} \\
K \pdnu u + u = h(\phi) &\quad \text{ on } \Sigma := \Gamma \times (0,\infty), \label{Rob} \\
\pd_t \phi - \LB \phi + f_\Gamma(\phi) + h'(\phi) \pdnu u = 0 &\quad \text{ on } \Sigma, \label{phi} \\
u(0) = u_0 &\quad \text{ in } \Omega, \\
\phi(0) = \phi_0 &\quad \text{ on } \Gamma,
\end{alignat}
\end{subequations}
where $K>0$ is a positive constant.
The study of problem \eqref{ACAC} is partially motivated by the following system of equations and dynamic boundary conditions of Allen--Cahn type for bulk variable $u$ and surface variable $\phi$ in a recent work \cite{CFL}:
\begin{subequations}\label{lim:ACAC}
\begin{alignat}{2}
\pd_t u - \Lap u + \beta(u) + \pi(u) \ni g_b &\quad  \text{ in } Q_T := \Omega \times (0,T), \\
\pd_t \phi - \LB \phi + \beta_\Gamma(\phi) + \pi_\Gamma(\phi) + \alpha \pdnu u \ni g_s &\quad  \text{ on } \Sigma_T := \Gamma \times (0,T), \\
u|_\Gamma = \alpha \phi + \eta &\quad  \text{ on } \Sigma_T, \\
u(0) = u_0 &\quad  \text{ in } \Omega,\\
\phi(0) = \alpha^{-1}(u_0 \vert_\Gamma - \eta) &\quad  \text{ on } \Gamma,
\end{alignat}
\end{subequations}
where $T > 0$ is an arbitrary but fixed constant, $g_b$ and $g_s$ are prescribed external forces, $\beta$ and $\beta_\Gamma$ are maximal monotone and possibly non-smooth graphs, while $\pi$ and $\pi_\Gamma$ are some non-monotone Lipschitz perturbations. The system \eqref{lim:ACAC} generalize previous analyses in \cite{CC,CF,CGNS,GG} from the standard relation $u \vert_{\Gamma} = \phi$ in dynamic boundary conditions to an affine linear transmission condition $u \vert_{\Gamma} = \alpha \phi + \eta$ with $\alpha \neq 0$ and $\eta \in \R$.  The motivation for such a modification to the relation between the bulk and surface variables, as originally described in \cite{CFL} for the Cahn--Hilliard equation, is to study potential competitions between bulk and surface phase separations in the special case $\alpha = -1$ and $\eta = 0$. Aside from a direct approach with an abstract formulation \cite{CF} to establish the strong well-posedness of \eqref{lim:ACAC}, the authors in \cite{CFL} appealed to the so-called boundary penalty method \cite{B,BE}, which employs a Robin boundary condition as a relaxation to approximate the Dirichlet-type boundary condition $u \vert_{\Gamma} =  \alpha \phi + \eta:=h(\phi)$. A first analysis was performed for the regularized system
\begin{subequations}\label{Rob:ACAC}
\begin{alignat}{2}
\pd_t u_K - \Lap u_K + \beta(u_K) + \pi(u_K) \ni g_b &\quad \text{ in } Q_T, \\
\pd_t \phi_K - \LB \phi_K + \beta_\Gamma(\phi_K) + \pi_\Gamma(\phi_K) + h'(\phi_K) \pdnu u_K \ni g_s &\quad  \text{ on } \Sigma_T, \label{Rob:b1}\\
K \pdnu u_K + u_K = h(\phi_K) &\quad  \text{ on } \Sigma_T, \\
u_K(0) = u_{K,0} &\quad  \text{ in } \Omega,\\
\phi_K(0) = \phi_{K,0} &\quad  \text{ on } \Gamma,
\end{alignat}
\end{subequations}
where $K>0$. We note that the variables $u_K$ and $\phi_K$ are coupled only through the Robin-type boundary condition $K\pdnu u_K + u_K = h(\phi_K)$ and the term $h'(\phi_K) \pdnu u_K$ in the surface equation \eqref{Rob:b1}.
For (possibly nonlinear) relations $h \in C^2(\R)$ with $h' \in W^{1,\infty}(\R)$, via a two-level approximation the existence of strong solutions to \eqref{Rob:ACAC} on $[0,T]$ is shown in \cite{CFL}. Moreover, for the special case $h(s) = \alpha s + \eta$, there exists a positive constant $C$ independent of $K$ such that
\begin{align*}
\norm{u_K - u}_{\mathbb{X}_\Omega} + \norm{\phi_K - \phi}_{\mathbb{X}_{\Gamma}} + K^{-1/2}\norm{\alpha \phi_K + \eta - u_K}_{L^2(\Sigma_T)} \leq C K^{1/2} \norm{\pdnu u}_{L^2(\Sigma_T)},
\end{align*}
where $(u_K, \phi_K)$ is the unique strong solution to \eqref{Rob:ACAC} and $(u,\phi)$ is the unique strong solution \eqref{lim:ACAC}, with $\mathbb{X}_{\Omega} := L^{\infty}(0,T;L^2(\Omega)) \cap L^2(0,T;H^1(\Omega))$ and $\mathbb{X}_\Gamma:= L^{\infty}(0,T;L^2(\Gamma)) \cap L^2(0,T;H^1(\Gamma))$.
In particular, the transmission condition $u|_\Gamma = \alpha \phi + \eta$ can be attained from the Robin relaxation \eqref{Rob:ACAC} at a linear rate in $K$.
It is also worth mentioning that the Robin type relaxation boundary condition has its own interest and appears in many other contexts, see for instance, the weak anchoring boundary condition for a bulk nematic liquid crystal with an included oil droplet \cite{KL} as well as the coupled bulk--surface system for receptor-ligand dynamics in cell biology \cite{ERV}.

In this contribution, after establishing the global well-posedness to problem \eqref{Rob:ACAC} (see Theorem \ref{thm:Exist}), our aim is to study its long-time behaviour as $t\to +\infty$.
The first attempt was made in \cite{CFL} that the authors gave a characterization on its $\omega$-limit set.
However, the structure of the $\omega$-limit remains unclear. In particular, for any initial datum can the $\omega$-limit set be just a singleton?
This issue is nontrivial since the nonconvexity of bulk and surface potentials indicate that the set of steady states may have a rather complicated structure.
Now writing $f(u) := \beta(u) + \pi(u)$, $f_\Gamma(\phi) := \beta_\Gamma(\phi) + \pi_\Gamma(\phi)$, with antiderivatives $F(u)$ and $F_\Gamma(\phi)$ such that $F'(u) = f(u)$ and $F_\Gamma'(\phi) = f_\Gamma(\phi)$, for any $K>0$, we reformulate \eqref{Rob:ACAC} into our problem \eqref{ACAC} (with zero external forces $g_b=g_s=0$ for the sake of simplicity).
 Then at least formally, we can deduce that it exhibits a Lyapunov structure that serves as a starting point of the analysis of long-time behavior:
\begin{align}\label{Lyap}
& \frac{d}{dt} \Big ( \int_\Omega \frac{1}{2} \abs{\nabla u}^2 + F(u) \dx + \int_\Gamma \frac{1}{2} \abs{\surf \phi}^2 + F_\Gamma(\phi) + \frac{1}{2K} \abs{u - h(\phi)}^2 \dH \Big ) \nonumber \\
& \quad + \norm{\pd_t u}_{L^2(\Omega)}^2 + \norm{\pd_t \phi}_{L^2(\Gamma)}^2 = 0\quad \forall\, t>0,
\end{align}
where $\surf g$ denotes the surface gradient of certain function $g$ defined on $\Gamma$. More precisely, assuming that the nonlinearities $f$, $f_\Gamma$ and $h$ are real analytic, we prove that every global strong solution $(u,\phi)$ of problem \eqref{ACAC} will converge to a single equilibrium $(u_*,\phi_*)$ as $t\to+\infty$ and moreover, we obtain a polynomial decay of the solution (See Theorem \ref{thm:Eqm}). The proof is based on the {\L}ojasiewicz--Simon approach \cite{S83}, which turns out to be an efficient method in the study of long-time behaviour of evolution equations with energy dissipation structure, see e.g., \cite{CGGM,GG,GG09,HJ,J,LW,Wu,WuZheng} and the references cited therein. For the Allen--Cahn system \eqref{lim:ACAC} with Dirichlet transmission condition $u|_\Gamma=\phi$, zero forcing terms $g_b = g_s = 0$ and analytic functions $\beta$, $\pi$, $\beta_\Gamma$, $\pi_\Gamma$, the convergence of its global solution $(u(t), \phi(t))$ to a single equilibrium as $t \to +\infty$ has been addressed in \cite{SW} (see also \cite{Wu07b} when the surface diffusion term $\Delta_\Gamma \phi$ is neglected). With minor modifications, similar conclusion can be draw for the affine linear case $u \vert_{\Gamma} = \alpha \phi + \eta$. However, for our problem \eqref{ACAC}, in order to overcome mathematical difficulties due to the bulk--surface coupling structure as well as the nonlinear Robin relaxation boundary condition, we have to derive a new type of gradient inequality of {\L}ojasiewicz--Simon type to achieve our goal (see Theorem \ref{thm:LSa}). Besides, it seems that the solution regularities established in \cite{CFL} are not sufficient for the study of long-time behaviour, partly as the previous estimates therein are not uniform with respect to the fixed terminal time $T > 0$, and partly due to the non-smooth maximal monotone graphs $\beta$ and $\beta_\Gamma$. Hence, for problem \eqref{ACAC}, we first need to establish new uniform-in-time regularity estimates for global solutions with analytical nonlinearities $h$, $f$ and $f_\Gamma$.

We note that the Allen--Cahn equation serves as possibly the simplest example for phase field models.
It will be interesting to perform corresponding analysis on the widely studied Cahn--Hilliard equation. While there are numerous contributions for the analysis of the Cahn--Hilliard equation with dynamic boundary conditions, amongst which we list the works \cite{CFP,CFch,CGS,Gal,GK,GMS,LW,MZ,Motoda,RZ,WuZheng}, to the best of our knowledge, analysis  of the Cahn--Hilliard system with general transmission relation $u \vert_\Gamma = h(\phi)$ between the bulk variable $u$ and the surface variable $\phi$ has not received much attention, neither has the corresponding relaxation with the Robin boundary condition. These will be the topics of our future study.

The remaining part of this paper is organized as follows.
In Section \ref{sec:main}, we introduce the functional settings and state the main results of this paper.
In Section \ref{sec:Exist}, we derive new regularity estimates and prove the global well-posedness of strong solutions to problem \eqref{ACAC}.
In Section \ref{sec:LS}, we establish an extended {\L}ojasiewicz--Simon inequality for our system with bulk--surface coupling structure.
In Section \ref{sec:Long}, we prove the convergence to equilibrium along with an estimate for the convergence rate.

\section{Main Results} \label{sec:main}
Throughout this paper, for a (real) Banach space $X$ we denote by $\|\cdot\|_X$ its norm, by $X'$ its
dual space, and by $\langle\cdot,\cdot\rangle_{X',X}$ the dual pairing between $X'$ and $X$.
The standard Lebesgue and Sobolev spaces in $\Omega$ are denoted by $L^p(\Omega)$ and $W^{k,p}(\Omega)$ for $p \in [1,\infty]$ and $k \geq 0$.  Likewise, $L^p(\Gamma)$ and $W^{k,p}(\Gamma)$ denote the corresponding Lebesgue and Sobolev spaces on $\Gamma$. In the case $p = 2$, we use the notation $H^k(\Omega) = W^{k,2}(\Omega)$ and $H^k(\Gamma) = W^{k,2}(\Gamma)$.
Let $I$ be an interval of $\mathbb{R}^+$ and $X$ a Banach space, the function space $L^p(I;X)$, $1 \leq p \leq \infty$ consists of $p$-integrable
functions with values in $X$. Moreover, $C(I;X)$ denotes the topological vector space of all bounded and continuous functions from $I$ to $X$, while
$W^{1,p}(I,X)$ $(1\leq p\leq \infty)$ stands for the space of all functions $u$ such that $u, \frac{\mathrm{d}u}{\dt}\in L^p(I;X)$, where
$\frac{\mathrm{d}u}{\dt}$ denotes the vector valued distributional derivative of $u$. For product function spaces, we shall make use of the notations
$$\bH := L^2(\Omega) \times L^2(\Gamma),\quad \bV := H^1(\Omega) \times H^1(\Gamma)\quad \text{and}\quad \bW := H^2(\Omega) \times H^2(\Gamma)$$
with equivalent norms
\begin{align*}
\norm{(u,v)}_{\bZ}^2 = \begin{cases} \norm{u}_{L^2(\Omega)}^2 + \norm{v}_{L^2(\Gamma)}^2 & \text{ if } \bZ = \bH, \\
\norm{u}_{H^1(\Omega)}^2 + \norm{v}_{H^1(\Gamma)}^2 & \text{ if } \bZ = \bV, \\
\norm{u}_{H^2(\Omega)}^2 + \norm{v}_{H^2(\Gamma)}^2 & \text{ if } \bZ = \bW.
\end{cases}
\end{align*}
Throughout the paper, $C\geq 0$ will stand for a generic constant and $\mathcal{Q}(\cdot)$ for a generic positive monotone increasing function.
Special dependence will be pointed out in the text if necessary.

Next, let us state the assumptions we shall work with:
\begin{enumerate}[label=$(\mathrm{A \arabic*})$, ref = $\mathrm{A \arabic*}$]
\item \label{ass:dom} $\Omega \subset \R^3$ is a bounded domain with smooth boundary $\Gamma$.
\item \label{ass:h} The function $h$ is analytic with $h', h'' \in L^{\infty}(\R)$ and satisfies for some positive constant $c > 0$, $\abs{h'''(s)} \leq c(1 + \abs{s}^q)$ holding for all $s \in \R$ with some exponent $q \in [0,\infty)$.
\item \label{ass:F} $F$ and $F_\Gamma$ are analytic functions and there exist positive constants $c_0, \dots, c_4$ such that
\begin{align*}
\abs{f''(s)} \leq c_0(1 + \abs{s}^{p}), \quad \abs{f_\Gamma''(s)} \leq c_0(1 + \abs{s}^q) &\quad \text{ for all } s \in \R, \\
F(s) \geq c_1 \abs{s} - c_2 , \quad F_\Gamma(s) \geq c_1 \abs{s} - c_2 &\quad \text{ for all } \abs{s} > c_3, \\
f'(s) \geq - c_4, \quad f_\Gamma'(s) \geq - c_4 &\quad \text{ for all } s \in \R,
\end{align*}
with exponents $p \in [0,3)$ and $q \in [0,\infty)$.
\item \label{ass:ini} The initial data satisfy $(u_0, \phi_0) \in \bW$ with the compatibility condition $K \pdnu u_0 + u_0 = h(\phi_0)$ holding a.e.~on $\Gamma$.
\end{enumerate}

Now we state the main results of this paper.

\begin{thm}[Global well-posedness]\label{thm:Exist}
Suppose that \eqref{ass:dom}--\eqref{ass:F} are satisfied.
For any initial data $(u_0, \phi_0)$ satisfying \eqref{ass:ini}, problem \eqref{ACAC} admits a unique global strong solution $(u,\phi)$ such that
\begin{align*}
 (u,\phi) \in C([0,+\infty);\bW) \quad (\pd_t u,\pd_t \phi) \in L^\infty(0,+\infty; \bH) \cap L^2(0,+\infty;\bV).
\end{align*}
Moreover, for any $\delta>0$, we have
\begin{align*}
(u,\phi)\in L^\infty(\delta, +\infty; H^3(\Omega)\times H^3(\Gamma)),\quad (\pd_t u,\pd_t \phi) \in L^\infty(\delta,+\infty; \bV).
\end{align*}
\end{thm}

\begin{thm}[Long-time behaviour]\label{thm:Eqm}
Under assumptions \eqref{ass:dom}--\eqref{ass:F}, for any initial data $(u_0, \phi_0) \in \bW$ satisfying \eqref{ass:ini}, the unique global strong solution to problem \eqref{ACAC} converges to an equilibrium $(u_*, \phi_*) \in \bW$, which is a strong solution to the stationary problem
\begin{subequations}\label{stat1}
\begin{alignat}{3}
-\Lap u_* + f(u_*) = 0 &\quad \text{ in } \Omega, \label{stat:u1} \\
K \pdnu u_* + u_* = h(\phi_*) &\quad \text{ on } \Gamma, \label{stat:Rob1} \\
- \LB \phi_* + f_\Gamma(\phi_*) + h'(\phi_*) \pdnu u_* = 0 &\quad \text{ on } \Gamma, \label{stat:phi1}
\end{alignat}
\end{subequations}
such that
\begin{align}\label{Rate}
\norm{(u(t), \phi(t)) - (u_*, \phi_*)}_{\bW} + \norm{(\pd_t u(t), \pd_t \phi(t))}_{\bH} \leq C (1+t)^{\frac{-\theta}{1-2\theta}},
\end{align}
where $\theta \in (0,\frac{1}{2})$ is a constant depending on $(u_*, \phi_*)$ and $C$ is a constant depending on $\|(u_0, \phi_0)\|_{\bW}$, $\Omega$, $\Gamma$ and coefficients of the system, but is independent of $t$.
\end{thm}

\section{Global Well-posedness} \label{sec:Exist}

In this section, we prove Theorem \ref{thm:Exist}. In view of \cite{CFL}, it remains to derive a series of uniform-in-time estimates that can be justified rigorously with the help of a Faedo--Galerkin approximation (using, in particular, the smoothness of the Galerkin coefficients with respect to the time variable due to the analytic nonlinearities to differentiate the equations in time, see \cite[\S 5 and 6]{CFL} for more details).

\subsection{A priori estimates}

\textbf{First estimate}. Testing \eqref{u} with $\pd_t u$, \eqref{phi} with $\pd_t \phi$, and using \eqref{Rob} yields the energy identity
\begin{align}\label{E1}
E(u(t), \phi(t)) + \int_0^t \norm{\pd_t u}_{L^2(\Omega)}^2 + \norm{\pd_t \phi}_{L^2(\Gamma)}^2 \ds = E(u_0, \phi_0) =: \E_0, \quad \forall\, t\geq 0
\end{align}
where
\begin{align}\label{energy}
E(u, \phi) = \int_\Omega \frac{1}{2} \abs{\nabla u}^2 + F(u) \dx + \int_\Gamma \frac{1}{2} \abs{\surf \phi}^2 + F_\Gamma(\phi) + \frac{1}{2K} \abs{u - h(\phi)}^2 \dH.
\end{align}
It is clear from the continuous embedding $\bW \subset L^{\infty}(\Omega) \times L^{\infty}(\Gamma)$ and assumptions \eqref{ass:F}--\eqref{ass:ini} that $\E_0$ is finite.
This yields
\begin{equation}\label{Est1}
\begin{aligned}
& \sup_{t\geq 0} E(u(t), \phi(t)) + \norm{\pd_t u}_{L^2(Q)}^2  + \norm{\pd_t \phi}_{L^2(\Sigma)}^2 \leq \E_0.
\end{aligned}
\end{equation}
Then by the assumption \eqref{ass:F}, it holds that
\begin{align}
\norm{u(t)}_{L^1(\Omega)} & \leq \frac{1}{c_1} \int_{\{x\in \Omega:\, |u(x,t)|>c_3\}} F(u(t)) \dx  + \frac{c_2}{c_1}|\Omega| + c_3 \abs{\Omega} \nonumber\\
& \leq \frac{1}{c_1} \int_{\Omega} F(u(t)) \dx +\left| \frac{1}{c_1} \int_{\{x\in\Omega:\, |u(x,t)|\leq c_3\}} F(u(t)) \dx\right|  + C \abs{\Omega} \nonumber\\
& \leq C(\E_0, \Omega, c_0, c_1, c_2, c_3,p) \quad \forall\, t\geq 0. \label{Est1a}
\end{align}
Then, we obtain with the help of the Gagliardo--Nirenberg inequality and \eqref{Est1}--\eqref{Est1a} that
\begin{align*}
\norm{u(t)}_{L^2(\Omega)} & \leq  C\norm{u(t)}_{L^1(\Omega)}^{\frac{2}{5}} \norm{\nabla u(t)}_{L^2(\Omega)}^{\frac{3}{5}} + C \norm{u(t)}_{L^1(\Omega)}\leq C, \quad \forall\, t\geq 0.
\end{align*}
By a similar argument, we have $\norm{\phi(t)}_{L^2(\Gamma)} \leq C$ for $t\geq 0$. Hence, combining with \eqref{Est1}, we get
\begin{align}\label{Est2}
\sup_{t\geq 0} \Big ( \norm{u(t)}_{H^1(\Omega)} + \norm{\phi(t)}_{H^1(\Gamma)} \Big ) \leq C.
\end{align}

\textbf{Second estimate}. As in \cite{CFL}, setting $\pd_t u(0) := \Lap u_0 - f(u_0)$ so that by assumptions \eqref{ass:F} and \eqref{ass:ini} we have $\pd_t u(0) \in L^2(\Omega)$.
Taking the time derivative of \eqref{u} and testing with $\pd_t u$, followed by integrating in time, gives rise to
\begin{align*}
& \frac{1}{2} \norm{\pd_t u(t)}_{L^2(\Omega)}^2 + \int_0^t \norm{\nabla \pd_t u}_{L^2(\Omega)}^2   + \frac{1}{K} \norm{\pd_t u}_{L^2(\Gamma)}^2 \ds \\
& \quad = \frac{1}{2} \norm{\pd_t u(0)}_{L^2(\Omega)}^2 + \int_0^t \Big ( \frac{1}{K} \int_\Gamma h'(\phi) \pd_t \phi \pd_t u \dH - \int_\Omega f'(u) \abs{\pd_t u}^2  \dx \Big ) \ds \\
& \quad \leq \frac{1}{2} \norm{\pd_t u(0)}_{L^2(\Omega)}^2 + \int_0^t \frac{1}{2K} \norm{\pd_t u}_{L^2(\Gamma)}^2 + \frac{\norm{h'}_{L^{\infty}(\R)}}{2K} \norm{\pd_t \phi}_{L^2(\Gamma)}^2 + c_4 \norm{\pd_t u}_{L^2(\Omega)}^2 \ds,
\end{align*}
where we have employed \eqref{ass:h} and \eqref{ass:F}.  In light of the previous estimate \eqref{Est1} we infer
\begin{align}\label{Est3}
\sup_{t\geq 0} \norm{\pd_t u(t)}_{L^2(\Omega)}^2 + \int_0^{+\infty} \norm{\pd_t u(t)}_{H^1(\Omega)}^2 \dt \leq C(\norm{u_0}_{H^2(\Omega)}, \E_0).
\end{align}
In a similar fashion, we set $\pd_t \phi(0) := \LB \phi_0 - f_\Gamma(\phi_0) - h'(\phi_0) \pdnu u_0$ so that $\pd_t \phi(0) \in L^2(\Gamma)$ holds by \eqref{ass:h}--\eqref{ass:ini}.  Then, taking the time derivative of \eqref{phi} and testing the resulting equation with $\pd_t \phi$ yields
\begin{align}
& \frac{1}{2} \norm{\pd_t \phi(t)}_{L^2(\Gamma)}^2 + \int_0^t \norm{\surf \pd_t \phi}_{L^2(\Gamma)}^2 + \frac{1}{K} \norm{\pd_t h(\phi)}_{L^2(\Gamma)}^2 \ds\nonumber \\
& \quad = \frac{1}{2} \norm{\pd_t \phi(0)}_{L^2(\Gamma)}^2 - \int_0^t \int_\Gamma \left(f_\Gamma'(\phi) \abs{\pd_t \phi}^2 + h''(\phi) \abs{\pd_t \phi}^2 \pdnu u - \frac{1}{K} \pd_t h(\phi) \pd_t u\right) \dH \ds\nonumber  \\
& \quad \leq \frac{1}{2} \norm{\pd_t \phi(0)}_{L^2(\Gamma)}^2 + \int_0^t c_4 \norm{\pd_t \phi}_{L^2(\Gamma)}^2 + \norm{h''}_{L^{\infty}(\R)} \norm{\pd_t \phi}_{L^4(\Gamma)}^2 \norm{\pdnu u}_{L^2(\Gamma)} \ds\nonumber \\
& \qquad + \int_0^t \frac{1}{2K} \norm{\pd_t h(\phi)}_{L^{2}(\Gamma)}^2 + \frac{1}{2K} \norm{\pd_t u}_{L^2(\Gamma)}^2\ds. \label{Est3a}
\end{align}
By the Gagliardo--Nirenburg inequality in two dimensions
\begin{align}\label{GN}
\norm{g}_{L^4(\Gamma)}^2 \leq C\norm{g}_{L^2(\Gamma)} \norm{\surf g}_{L^2(\Gamma)} + C\norm{g}_{L^2(\Gamma)}^2,
\end{align}
and on account of the boundedness of $\pdnu u = K^{-1}(h(\phi) - u)$ in $L^{\infty}(0,+\infty;L^2(\Gamma))$ from \eqref{Est1} and \eqref{ass:h}, we find that
\begin{align*}
\int_0^t  \norm{h''}_{L^{\infty}(\R)} \norm{\pd_t \phi}_{L^4(\Gamma)}^2 \norm{\pdnu u}_{L^2(\Gamma)} \ds \leq \int_0^t \frac{1}{2} \norm{\surf \pd_t \phi}_{L^2(\Gamma)}^2 + C \norm{\pd_t \phi}_{L^2(\Gamma)}^2 \ds,
\end{align*}
and so, combining with \eqref{Est1}, \eqref{Est3}, we infer from \eqref{Est3a} that
\begin{align}\label{Est4}
&\sup_{t\geq 0} \norm{\pd_t \phi(t)}_{L^2(\Gamma)}^2 + \int_0^{+\infty} \norm{\pd_t \phi(t)}_{H^1(\Gamma)}^2 +\norm{\pd_t h(\phi(t))}_{L^2(\Gamma)}^2 \dt\nonumber\\
 &\quad \leq C(\norm{(u_0, \phi_0)}_{\bW}, \E_0).
\end{align}
As a consequence of \eqref{Rob}, \eqref{Est3} and \eqref{Est4}, we can further deduce that
\begin{align}\label{Est5}
\int_0^{+\infty}\norm{\pd_t \pdnu u(t)}_{L^2(\Gamma)}^2 \dt \leq C.
 \end{align}

\textbf{Third estimate}.  From \eqref{Est1}, \eqref{Est2}, \eqref{Est4}, \eqref{ass:h} and \eqref{ass:F}, we claim that
\begin{align*}
\sup_{t\geq 0} \norm{\pd_t \phi(t) + f_\Gamma(\phi(t)) + h'(\phi(t)) K^{-1} (h(\phi(t)) - u(t))}_{L^2(\Gamma)} \leq C.
\end{align*}
Indeed, the assertion for $\pd_t \phi$ comes from \eqref{Est4}, while using \eqref{ass:h}, \eqref{ass:F} and the Sobolev embedding $H^1(\Gamma) \subset L^r(\Gamma)$ for any $r \in [1, +\infty)$,
\begin{align*}
\norm{f_\Gamma(\phi)}_{L^2(\Gamma)}^2 & \leq C(1 + \norm{\phi}_{H^1(\Gamma)}^{2(q+1)}), \\
\norm{h'(\phi)(h(\phi) - u)}_{L^2(\Gamma)}^2 & \leq (1 + \norm{\phi}_{L^2(\Gamma)}^2 + \norm{u}_{H^1(\Omega)}^2).
\end{align*}
Hence, applying regularity theory to \eqref{phi} viewed as an elliptic equation for $\phi$ leads to
\begin{align*}
\norm{\phi(t)}_{H^2(\Gamma)}\leq C\Big ( \norm{\pd_t \phi + f_\Gamma(\phi) + h'(\phi) K^{-1} (h(\phi) - u)}_{L^2(\Gamma)} + \norm{\phi}_{L^2(\Gamma)} \Big),
\end{align*}
and thus we arrive at
\begin{align}\label{Est6}
\sup_{t\geq 0} \norm{\phi(t)}_{H^2(\Gamma)} \leq C.
\end{align}
Meanwhile, for the bulk variable $u$, we aim to apply a similar argument to the elliptic problem \eqref{u}--\eqref{Rob}.
By the Lipschitz continuity of $h$ and \eqref{Est1}, we first note that $h(\phi) \in L^{\infty}(0,+\infty;H^1(\Gamma))$, and so together with $u \in L^{\infty}(0,+\infty;H^1(\Omega))$ we infer
\begin{align*}
\norm{\pdnu u(t)}_{H^{\frac{1}{2}}(\Gamma)} \leq C \norm{\phi(t)}_{H^1(\Gamma)} + C\norm{u(t)}_{H^1(\Omega)} \leq C, \quad \forall\, t\geq 0.
\end{align*}
On the other hand, according to \eqref{ass:F} we have
\begin{align*}
\norm{f(u)}_{L^2(\Omega)}^2 \leq C \Big (1 + \norm{u}_{L^{2p+4}(\Omega)}^{2p+4} \Big)  \quad \text{ for } p \in [0,3).
\end{align*}
For exponents $p \in (1,3)$, the Gagliardo--Nirenberg inequality in three dimensions
\begin{align*}
\norm{u}_{L^{2p+4}(\Omega)} \leq C \norm{u}_{H^2(\Omega)}^{\alpha} \norm{u}_{L^6(\Omega)}^{1-\alpha}\quad  \text{ for } \alpha = \frac{1}{2} - \frac{3}{2p+4} \in (0,1),
\end{align*}
leads to the estimate
\begin{align*}
\norm{f(u)}_{L^2(\Omega)}^2 \leq C \Big ( 1+ \norm{u}_{H^2(\Omega)}^{p-1} \Big ),
\end{align*}
where we used fact that $\alpha(2p+4) = p-1 \in (0,2)$.  Hence, together with \eqref{Est3}, we infer from the elliptic regularity theory and Young's inequality that
\begin{align*}
\norm{u}_{H^2(\Omega)}^2 &\leq C \Big (\norm{\pd_t u}_{L^2(\Omega)}^2 + \norm{f(u)}_{L^2(\Omega)}^2 + \norm{\pdnu u}_{H^{\frac{1}{2}}(\Gamma)}^2 \Big )\\
& \leq C + \frac{1}{2} \norm{u}_{H^2(\Omega)}^2,
\end{align*}
for the case $p \in(1,3)$, which results in
\begin{align}\label{Est7}
\sup_{t\geq 0} \norm{u(t)}_{H^2(\Omega)} \leq C.
\end{align}
For exponents $p \in [0,1]$, the situation is easier such that the Sobolev embedding $H^1(\Omega) \subset L^6(\Omega)$ and the estimate \eqref{Est1} implies $f(u) \in L^{\infty}(0,+\infty; L^2(\Omega))$, leading immediately to the same regularity assertion \eqref{Est7}.\medskip

\textbf{Fourth estimate}.
Employing \eqref{ass:F}, the Sobolev embedding $H^1(\Gamma) \subset L^r(\Gamma)$ for any $r \in [1, +\infty)$ and \eqref{Est1},
\begin{align*}
\norm{f_\Gamma'(\phi) \pd_t \phi}_{L^2(\Gamma)}^2 & \leq C\int_\Gamma \Big (1 + \abs{\phi}^{2(q+1)} \Big ) \abs{\pd_t \phi}^2 \dH\\
&  \leq C \Big ( \norm{\pd_t \phi}_{L^2(\Gamma)}^2 + \norm{\phi}_{L^{6(q+1)}(\Gamma)}^{2(q+1)} \norm{\pd_t \phi}_{L^3(\Gamma)}^{2} \Big )\\
&  \leq C \norm{\pd_t \phi}_{H^1(\Gamma)}^2,
\end{align*}
while using \eqref{Rob}, the Lipschitz continuity of $h$, \eqref{Est1} and \eqref{Est7},
\begin{align*}
\norm{\pd_t \phi \pdnu u}_{L^2(\Gamma)}^2 & \leq C \norm{\pd_t \phi}_{L^4(\Gamma)}^2 \Big (\norm{h(\phi) - h(0)}_{L^4(\Gamma)}^2 + \norm{h(0)}_{L^4(\Gamma)}^2 + \norm{u}_{L^4(\Gamma)}^2 \Big ) \\
& \leq C \norm{\pd_t \phi}_{L^4(\Gamma)}^2 \Big ( 1 + \norm{\phi}_{L^4(\Gamma)}^2 + \norm{u}_{L^4(\Gamma)}^2 \Big ) \\
& \leq C \norm{\pd_t \phi}_{H^1(\Gamma)}^2.
\end{align*}
Then taking the time derivative of the surface equation \eqref{phi}, and testing the resultant with $t \pd_{tt} \phi$ yields
\begin{align}
&  \frac{1}{2}  \frac{d}{dt} \Big (t\norm{\surf \pd_t \phi}_{L^2(\Gamma)}^2 \Big ) - \frac{1}{2} \norm{\surf \pd_t \phi}_{L^2(\Gamma)}^2 + t \norm{\pd_{tt} \phi}_{L^2(\Gamma)}^2 \nonumber \\
& \quad = - t \int_\Gamma \Big [f_\Gamma'(\phi) \pd_t \phi  + h''(\phi) \pd_t \phi \pdnu u + K^{-1} h'(\phi) (h'(\phi) \pd_t \phi - \pd_t u) \Big ]\pd_{tt} \phi \dH \nonumber\\
& \quad \leq \frac{t}{2} \norm{\pd_{tt} \phi}_{L^2(\Gamma)}^2 + Ct \Big (\norm{f_\Gamma'(\phi) \pd_t \phi}_{L^2(\Gamma)}^2 + \norm{\pd_t \phi \pdnu u}_{L^2(\Gamma)}^2 +  \norm{\pd_t \phi}_{L^2(\Gamma)}^2 + \norm{\pd_t u}_{L^2(\Gamma)}^2 \Big ) \nonumber\\
& \quad \leq \frac{t}{2} \norm{\pd_{tt} \phi}_{L^2(\Gamma)}^2 + Ct \norm{(\pd_t u,\pd_t \phi)}_{\bV}^2, \label{Est8aa}
\end{align}
where we have used the estimates \eqref{Est6} and \eqref{Est7}.
Integrating \eqref{Est8aa} in $t$ from $0$ to some $s > 0$, then dividing the resulting inequality by $s$ and employing \eqref{Est3} and \eqref{Est4}, we have
\begin{align}\label{Est8a}
& \norm{\surf \pd_t \phi(s)}_{L^2(\Gamma)}^2 + \frac{1}{s} \int_0^s t \norm{\pd_{tt} \phi}_{L^2(\Gamma)}^2 \dt \nonumber\\
&\quad  \leq \frac{1}{s} \int_0^s \left(Ct\norm{(\pd_t u, \pd_t \phi)}_{\bV}^2 +  \norm{\surf \pd_t \phi}_{L^2(\Gamma)}^2\right) \dt \nonumber \\
&\quad \leq C\left(1 + \frac{1}{s}\right) \quad \forall\, s>0.
\end{align}
Similarly, taking the time derivative of \eqref{u}--\eqref{Rob} and testing the resultant with $t \pd_{tt} u$, after integrating from $0$ to $s$, we obtain
\begin{align}
& \frac{1}{2}  \Big ( s \norm{\nabla \pd_t u(s)}_{L^2(\Omega)}^2 + K^{-1} s \norm{\pd_t u(s)}_{L^2(\Gamma)}^2 \Big ) + \int_0^s t \norm{ \pd_{tt} u}_{L^2(\Omega)}^2 \dt \nonumber\\
& \quad =  \int_0^s \frac{1}{2} \Big ( \norm{\nabla \pd_t u}_{L^2(\Omega)}^2 + K^{-1} \norm{\pd_t u}_{L^2(\Gamma)}^2 \Big ) \dt\nonumber \\
& \qquad - \int_0^s  t \int_\Omega f'(u) \pd_t u \pd_{tt} u \dx \dt + K^{-1}\int_0^s  t \int_\Gamma h'(\phi) \pd_ t \phi \pd_{tt} u \dH  \dt. \label{Est8b}
\end{align}
From \eqref{Est3} the first integral on the right-hand side of \eqref{Est8b} is uniformly bounded in $s \in (0,\infty)$.
Meanwhile, we infer from \eqref{Est7} and the Sobolev embedding theorem that $\|u(t)\|_{L^{\infty}(\Omega)}$ is uniformly for $t\geq 0$, and by the continuity of $f'(\cdot)$ this further implies that $\|f'(u(t))\|_{L^{\infty}(\Omega)}$ is also uniformly bounded in time.  Hence,
\begin{align*}
\abs{\int_0^s t \int_\Omega f'(u) \pd_t u \pd_{tt} u \dx \dt} \leq \frac{1}{2} \int_0^s t \norm{\pd_{tt} u}_{L^2(\Omega)}^2 \dt + C \int_0^s t \norm{\pd_t u}_{L^2(\Omega)}^2 \dt.
\end{align*}
For the surface integral, we perform an integration by parts and get
\begin{align*}
& \abs{\int_0^s t \int_\Gamma h'(\phi) \pd_t \phi \pd_{tt} u \dH \dt} \\
& \quad = \left | - \int_0^s \int_\Gamma t( h'(\phi) \pd_{tt} \phi \pd_t u + h''(\phi) (\pd_t \phi)^2 \pd_t u) + h'(\phi) \pd_t \phi \pd_t u \dH \dt \right . \\
& \quad \qquad  \left. + \int_\Gamma s h'(\phi(s)) \pd_t \phi(s) \pd_t u(s) \dH \right | \\
& \quad \leq C \int_0^s t \Big ( \norm{\pd_{tt} \phi}_{L^2(\Gamma)} + \norm{\pd_t \phi}_{L^4(\Gamma)}^2 \Big ) \norm{\pd_t u }_{L^2(\Gamma)} + \norm{\pd_t \phi}_{L^2(\Gamma)} \norm{\pd_t u}_{L^2(\Gamma)} \dt \\
& \qquad + C s \norm{\pd_t \phi(s)}_{L^2(\Gamma)} \norm{\pd_t u(s)}_{L^2(\Gamma)}.
\end{align*}
Then, dividing \eqref{Est8b} by $s$ and employing the above estimates lead to
\begin{align*}
& \norm{\nabla \pd_t u(s)}_{L^2(\Omega)}^2 + K^{-1} \norm{\pd_t u(s)}_{L^2(\Gamma)}^2 + \frac{1}{s} \int_0^s t \norm{\pd_{tt} u}_{L^2(\Omega)}^2 \dt \\
& \quad \leq \frac{C}{s} + \frac{C}{s} \int_0^s t \Big (\norm{\pd_t u}_{L^2(\Omega)}^2 + \norm{\pd_{tt} \phi}_{L^2(\Gamma)}^2 +\norm{\pd_t u}_{L^2(\Gamma)}^2+ \norm{\pd_{t} \phi}_{L^4(\Gamma)}^2 \norm{\pd_{t} u}_{L^2(\Gamma)} \Big ) \dt \\
& \qquad + \frac{C}{s} \int_0^s \norm{\pd_{t} \phi}_{L^2(\Gamma)}^2 + \norm{\pd_{t} u}_{L^2(\Gamma)}^2 \dt + C \Big (\norm{\pd_t \phi(s)}_{L^2(\Gamma)}^2 + \norm{\pd_t u(s)}_{L^2(\Gamma)}^2 \Big ) \\
& \quad \leq \frac{C}{s} \Big ( 1 + \int_0^s t \norm{\pd_{tt} \phi}_{L^2(\Gamma)}^2 \dt \Big ) + C \int_0^s \norm{\pd_{t} u}_{H^1(\Omega)}^2\dt \nonumber\\
& \qquad + \sup_{t\geq 0}\norm{\pd_t u}_{L^2(\Gamma)} \int_0^s \norm{\pd_{t} \phi}_{H^1(\Gamma)}^2 \dt
 + C\sup_{t\geq 0}\left(\norm{\pd_t \phi(t)}_{L^2(\Gamma)}^2 + \norm{\pd_t u(t)}_{L^2(\Gamma)}^2\right) \\
& \quad \leq C \Big ( 1 + \frac{1}{s} \Big ) \quad \forall\, s>0,
\end{align*}
by virtue of \eqref{Est3}, \eqref{Est4} and \eqref{Est8a}.  The above estimate, together with \eqref{Est8a}, leads to
\begin{align}\label{Est8}
\norm{\pd_t u(t)}_{H^1(\Omega)}^2 + \norm{\pd_t \phi(t)}_{H^1(\Gamma)}^2 \leq C \Big ( 1 + \frac{1}{\delta} \Big )\quad \forall\, t\geq \delta>0.
\end{align}

\textbf{Fifth estimate}. Next, we check that for all $t\geq 0$, it holds
\begin{align*}
\norm{\surf f_\Gamma(\phi)}_{L^2(\Gamma)}^2
& \leq C \int_\Gamma (1 + \abs{\phi}^{2(q+1)} ) \abs{\surf \phi}^2 \dH\nonumber\\
& \leq C \norm{\phi}_{H^1(\Gamma)}^2 + C\norm{\phi}_{L^{4q+4}(\Gamma)}^{2q+2}\norm{\surf \phi}_{L^4(\Gamma)}^2 \\
& \leq C \norm{\phi}_{H^1(\Gamma)}^2 + C\norm{\phi}_{H^2(\Gamma)}^{2q+4} \leq C, \\[2ex]
\norm{h''(\phi) (h(\phi) - u) \surf \phi }_{L^2(\Gamma)}^2 & \leq C \norm{\surf \phi}_{L^4(\Gamma)}^2 \Big ( 1 + \norm{\phi}_{L^4(\Gamma)}^2 + \norm{u}_{L^4(\Gamma)}^2 \Big ) \\
&  \leq C \norm{\phi}_{H^2(\Gamma)}^2 \Big ( 1 + \norm{\phi}_{H^1(\Gamma)}^2 + \norm{u}_{H^2(\Omega)}^2 \Big ) \leq C, \\[2ex]
\norm{h'(\phi) (h''(\phi) \surf \phi - \surf u)}_{L^2(\Gamma)}^2 &  \leq C \norm{\surf \phi}_{L^2(\Gamma)}^2 + C\norm{\surf u}_{L^2(\Gamma)}^2 \\
&  \leq C \norm{\phi}_{H^1(\Gamma)}^2 + C\norm{u}_{H^2(\Omega)}^2  \leq C,
\end{align*}
where the second last inequality comes from the trace theorem $H^2(\Omega) \hookrightarrow H^{\frac{3}{2}}(\Gamma)$ and the continuous embedding $H^{\frac{3}{2}}(\Gamma)  \subset H^1(\Gamma)$.
Then, together with \eqref{Est8} and the elliptic regularity theory for $\phi$ we deduce that
\begin{align}\label{Est9}
&\norm{\phi(t)}_{H^3(\Gamma)} \nonumber\\
&\quad \leq C \norm{\pd_t \phi(t) + f_\Gamma(\phi(t)) + K^{-1} h'(\phi(t))(h(\phi(t)) - u(t))}_{H^1(\Gamma)}+C\norm{\phi(t)}_{L^2(\Gamma)}\nonumber\\
&\quad \leq C_\delta \quad \forall\, t\geq \delta>0,
\end{align}
where the constant $C_\delta>0$ is independent of $t$ but it will tends to $+\infty$ as $\delta \to 0^+$. In a similar fashion, we use \eqref{ass:h} to deduce that $h(\phi) \in L^{\infty}(0,+\infty; H^2(\Gamma))$, and the trace theorem to conclude that $u \in L^\infty(0,+\infty; H^{\frac{3}{2}}(\Gamma))$.  Then, the relation \eqref{Rob} yields
\begin{align*}
\norm{\pdnu u(t)}_{H^{\frac{3}{2}}(\Gamma)} \leq  C\norm{h(\phi(t))}_{H^{2}(\Gamma)} + C\norm{u(t)}_{H^2(\Omega)}  \leq C,
\quad \forall\, t\geq 0.
\end{align*}
Moreover, according to \eqref{ass:F} and \eqref{Est7}, we see that
\begin{align*}
\norm{\nabla f(u)}_{L^2(\Omega)}^2
& \leq C \int_\Omega (1 + \abs{u}^{2p+2} ) \abs{\nabla u}^2 \dx \nonumber\\
& \leq C \norm{u}_{H^1(\Omega)}^2 + C \norm{u}_{L^{3p+3}(\Omega)}^{2p+2} \norm{\nabla u}_{L^6(\Omega)}^2 \leq C \quad \forall\, t\geq 0.
\end{align*}
Then, the elliptic regularity theory, \eqref{Est7} and \eqref{Est8} imply that
\begin{align}\label{Est10}
\norm{u(t)}_{H^3(\Omega)} &\leq C\left( \norm{\pd_t u(t)}_{H^1(\Omega)} + \norm{f(u(t))}_{H^1(\Omega)} + \norm{\pdnu u(t)}_{H^{\frac{3}{2}}(\Gamma)} +\norm{u(t)}_{L^2(\Omega)}\right)\nonumber\\
 &\leq C_\delta \quad \forall\, t\geq \delta>0.
\end{align}

\subsection{Proof of Theorem \ref{thm:Exist}}
Based on the estimates \eqref{Est1}, \eqref{Est2}, \eqref{Est3}, \eqref{Est4}, \eqref{Est6}, \eqref{Est7}, \eqref{Est8}, \eqref{Est9}, \eqref{Est10}, the existence of a global strong solution to problem \eqref{ACAC} with required regularities can be proved in a standard manner, using a similar Galerkin approximation scheme devised in \cite{CFL}. Moreover, using \eqref{Est3}, \eqref{Est4}, \eqref{Est6}, \eqref{Est7}, \eqref{Est8}, the elliptic regularity theorem for $(u,\phi)$, we can show that for arbitrary $T\in (0,+\infty)$, $(u, \phi)\in L^2(0,T; H^3(\Omega)\times H^3(\Gamma))$.
Then by the continuous embedding $L^2(0,T;H^3(\Omega) \times H^3(\Gamma)) \cap H^1(0,T;\bV) \subset C([0,T];\bW)$ (see e.g., \cite[Chapter 1, Theorem 3.1]{LM})
 and $T>0$ is arbitrary, we have $(u, \phi)\in C([0,+\infty); \bW)$.
 Next, by the same energy method as in \cite[\S 4]{CFL} and some minor modifications due to assumption \eqref{ass:F}, we are able to derive a continuous dependence result on initial data. More precisely, let $(u_1, \phi_1)$ and $(u_2, \phi_2)$ denote two strong solutions to problem \eqref{ACAC} corresponding to initial data $(u_{0,1}, \phi_{0,1})$ and $(u_{0,2}, \phi_{0,2})$, respectively, it holds
\begin{align*}
&\norm{(u_1(t)-u_2(t), \phi_1(t)-\phi_2(t))}_{\bH}^2 + \int_0^t \norm{(u_1-u_2, \phi_1-\phi_2)}_{\bV}^2 \ds \nonumber\\
&\quad \leq C e^{C t} \norm{(u_{0,1}-u_{0,2}, \phi_{0,1}-\phi_{0,2})}_{\bH}^2, \quad \forall\, t>0,
\end{align*}
for some positive constant $C$ depending on the initial data, $\Omega$, $\Gamma$, but not on $u$, $\phi$ and $t$.
Then the uniqueness of strong solutions easily follows.

The proof of Theorem \ref{thm:Exist} is complete.

\section{Extended {\L}ojasiewicz--Simon Inequality}\label{sec:LS}
In this section, our aim is to establish an extended {\L}ojasiewicz--Simon inequality, which plays a crucial role in the study of long-time behaviour for the bulk--surface coupled Allen--Cahn system \eqref{ACAC}.

From assumptions \eqref{ass:dom}--\eqref{ass:F}, it is straightforward to verify that the energy functional $E$ is continuously Fr\'echet differentiable on $\bV$.
For any $(u,\phi),\, (w, \xi) \in \bV$, we define $E'=M: \bV \to \bV'$ by
\begin{align}
\inner{M(u,\phi)}{(w,\xi)}_{\bV',\bV} & =  \left. \frac{d E(u+\eps w, \phi + \eps \xi)}{d \eps} \right \vert_{\eps = 0}\nonumber \\
& =  \int_\Omega \nabla u \cdot \nabla w + f(u) w \dx + \int_\Gamma \surf \phi \cdot \surf \xi + f_\Gamma(\phi) \xi \dH \nonumber\\
& \quad + \int_\Gamma \frac{1}{K} (u - h(\phi))( w - h'(\phi) \xi) \dH. \label{defn:MM}
\end{align}
We say that $(u_*,\phi_*)\in \bV$ is a critical point of $E(u,\phi)$ if $E'(u_*, \phi_*)=0$. Consider the stationary problem
\begin{subequations}\label{stat}
\begin{alignat}{3}
-\Lap u_* + f(u_*) = 0 &\quad  \text{ in } \Omega, \label{stat:u} \\
K \pdnu u_* + u_* = h(\phi_*) &\quad  \text{ on } \Gamma, \label{stat:Rob} \\
- \LB \phi_* + f_\Gamma(\phi_*) + h'(\phi_*) \pdnu u_* = 0 &\quad  \text{ on } \Gamma. \label{stat:phi}
\end{alignat}
\end{subequations}
Then we prove the following result that gives the equivalence between the critical points of $E$ and the solutions of problem \eqref{stat}.
\begin{prop}\label{thm:stat}
If $(u_*,\phi_*) \in \bW$ is a strong solution to the stationary problem \eqref{stat}, then $(u_*, \phi_*)$ is a critical point to the functional $E$, i.e., $E'(u_*, \phi_*) = 0$ as an equality in $\bV'$.  Conversely, if $(u_*, \phi_*)$ is a critical point to the functional $E$, then $(u_*, \phi_*)\in \bW$ is a strong solution to the stationary problem \eqref{stat}.
\end{prop}
\begin{proof}
If $(u_*, \phi_*)$ satisfies the stationary problem, then for any $(w, \xi) \in \bV$ we have
\begin{align*}
\int_\Omega (- \Lap u_* + f(u_*)) w \dx + \int_\Gamma (- \LB \phi_* + f_\Gamma(\phi_*) + h'(\phi_*) \pdnu u_*) \xi \dH = 0.
\end{align*}
Integrating by parts and applying the Robin boundary condition \eqref{stat:Rob} for $u_*$ yields
\begin{align}
0 & = \int_\Omega \nabla u_* \cdot \nabla w + f(u_*) w \dx + \int_\Gamma \surf \phi_* \cdot \surf \xi + f_\Gamma(\phi_*) \xi \dH\nonumber \\
& \quad + \int_\Gamma \frac{1}{K} (u_* - h(\phi_*))(v - h'(\phi_*) \xi) \dH\nonumber \\
&  = \inner{E'(u_*, \phi_*)}{(w,\xi)}_{\bV',\bV}.\label{stat:weak}
\end{align}
Hence, $(u_*, \phi_*)$ is a critical point of $E$.

Conversely, if $(u_*, \phi_*)$ is a critical point of $E$, then $(u_*, \phi_*) \in \bV$ is a weak solution to the stationary problem \eqref{stat}.  Substituting $w = 0$ in \eqref{stat:weak} yields the weak formulation of the elliptic equation
\begin{align*}
- \LB \phi_* = - f_\Gamma(\phi_*) - K^{-1} h'(\phi_*)(h(\phi_*) - u_*) =: g_\Gamma \quad \text{ on } \Gamma.
\end{align*}
Using \eqref{ass:h}, the trace theorem $H^1(\Omega) \hookrightarrow L^2(\Gamma)$, the growth assumption \eqref{ass:F} for $f_\Gamma$ and the Sobolev embedding $H^1(\Gamma) \subset L^{r}(\Gamma)$ for any $r \in [1,+\infty)$, we find that the right-hand side $g_\Gamma$ is bounded in $L^2(\Gamma)$.  Then, the elliptic regularity theory yields that $\phi_*$ is bounded in $H^2(\Gamma)$, and by \eqref{ass:h} one can see that $h(\phi_*) \in H^1(\Gamma) \subset H^{\frac{1}{2}}(\Gamma)$. Then, substituting $\xi = 0$ in \eqref{stat:weak} yields the weak formulation of the elliptic problem
\begin{align*}
- \Lap u_* = - f(u_*) =: g_\Omega &\quad  \text{ in } \Omega, \\
\pdnu u_*  = K^{-1}(h(\phi_*) - u_*) =: \tilde g_\Gamma & \quad \text{ on } \Gamma,
\end{align*}
where from the above discussion it holds that $\tilde g_\Gamma \in H^{\frac{1}{2}}(\Gamma)$.  By \eqref{ass:F} and the embedding $H^1(\Omega) \subset L^6(\Omega)$ we see that for exponents $p \in [0,1]$,
\begin{align*}
\norm{f(u)}_{L^2(\Omega)} \leq C( 1 + \norm{u}_{L^{2p+4}(\Omega)}^{p+2}) \leq C(1 + \norm{u}_{H^1(\Omega)}^{p+2}) \leq C,
\end{align*}
and so $g_\Omega \in L^2(\Omega)$.  By the elliptic regularity theory we obtain $u_* \in H^2(\Omega)$.   For exponents $p \in (1,3)$, we follow a similar argument as in the derivation of \eqref{Est7} to deduce that $u_* \in H^2(\Omega)$.  Therefore, $(u_*, \phi_*) \in \bW$ is a strong solution to the stationary problem \eqref{stat}.

The proof is complete.
\end{proof}

Now we state the main result of this section.
\begin{thm}[Extended {\L}ojasiewicz--Simon inequality]\label{thm:LSa}
Suppose that \eqref{ass:dom}--\eqref{ass:F} are satisfied. Let $(u_*,\phi_*)\in \bW$ be a critical point of the energy functional $E(u,\phi)$ defined by \eqref{energy}.
There exist constants $\theta \in (0,\frac{1}{2})$ and $\beta > 0$ depending on $(u_*, \phi_*)$, such that for any $(u,\phi) \in \bV$ satisfying $\norm{(u, \phi) - (u_*, \phi_*)}_{\bV} < \beta$, we have
\begin{align}\label{LS:alt}
\norm{M(u, \phi)}_{\bV'} \geq \abs{E(u, \phi) - E(u_*, \phi_*)}^{1-\theta}.
\end{align}
\end{thm}

The proof of Theorem \ref{thm:LSa} is based along the procedure in \cite{HJ}, see in particular \cite{SW} for the modified argument that is valid for nonlinear dynamic boundary conditions.
However, in our current case, some new difficulties due to the bulk--surface coupling and the nonlinear Robin type boundary condition have to be handled.

For any critical points $(u_*, \phi_*) \in \bW$ (cf. Proposition \ref{thm:stat}) of $E$, we consider perturbation functions $(v,\psi)$ and write
\begin{align*}
u = u_* + v, \quad \phi = \phi_* + \psi, \quad \E(v,\psi)  := E(u, \phi) := E(u_* + v, \phi_* + \psi).
\end{align*}
We also set
\begin{align*}
\M(v,\psi) = M(u,\phi) = M(u_*+v, \phi_* + \psi),
\end{align*}
so that from the definition of a critical point of $E$, we have $\M(0,0) = M(u_*, \phi_*) = 0$. Keeping the above notations in mind, it remains to prove that:
for a given critical point $(u_*, \phi_*)$ of $E$, there exist constants $\theta \in (0,\frac{1}{2})$ and $\beta > 0$ depending on $(u_*, \phi_*)$, such that for any $(v,\psi) \in \bV$ satisfying $\norm{(v,\psi)}_{\bV} < \beta$, it holds
\begin{align}\label{LS}
\norm{\M(v,\psi)}_{\bV'} \geq \abs{\E(v,\psi) - \E(0,0)}^{1-\theta}.
\end{align}

The proof of Theorem \ref{thm:LSa} consists of several steps.

\paragraph{Step 1: Analysis of a certain linear operator.}
We define the strictly positive, self-adjoint and unbounded operator $A_\Gamma := - \LB + I$ from $D(A_\Gamma) = H^2(\Gamma)$ to $L^2(\Gamma)$.  Then, standard spectral theory yields the existence of a complete orthonormal basis $\{y_j\}_{j \in \N} \subset D(A_\Gamma)$ in $L^2(\Gamma)$, along with an ordered sequence of eigenvalues $0 < \mu_1 \leq \mu_2 \leq \cdots$ satisfying $\mu_j \to \infty$ as $j \to \infty$, such that for all $j \in \N$,
\begin{align*}
A_\Gamma y_j = \mu_j y_j.
\end{align*}
Next, following \cite[\S 5]{CFL}, we define the Hilbert space $\tilde V$ and associated inner product:
\begin{align*}
\tilde V := \{(a,b) \in H^1(\Omega) \times H^{1/2}(\Gamma) \, : \, a \vert_\Gamma = b\}, \quad (\bm{p}, \bm{q})_{\tilde V} = (p,q)_{H^1(\Omega)}
\end{align*}
for $\bm{p} = (p, p \vert_\Gamma)$ and $\bm{q} = (q, q \vert_\Gamma)$.  Then, it is shown that the abstract operator $A_\Omega : \tilde V \to (\tilde V)'$ defined by
\begin{align*}
\inner{A_\Omega \bm{p}}{\bm{q}}_{(\tilde V)',\tilde V} = \int_\Omega \nabla p \cdot \nabla q \dx + \int_\Gamma K^{-1} p \vert_\Gamma \, q \vert_\Gamma \dH
\end{align*}
is strictly positive, self-adjoint, coercive on $\tilde V$ with compact inverse.  Hence, by standard spectral theory there exists an ordered sequence of eigenvalues $0 < \lambda_1 \leq \lambda_2 \leq \cdots$ satisfying $\lambda_j \to \infty$ as $j \to \infty$, and a corresponding sequence of eigenfunctions $\{w_j\}_{j \in \N}$ that forms an orthonormal basis in $\bH$ satisfying
\begin{equation}\label{ortho}
\begin{aligned}
(\bm{w}_i, \bm{w}_j)_{\bH} = ( (w_i, w_i \vert_\Gamma), (w_j, w_j \vert_\Gamma) )_{\bH} & = \delta_{ij}, \\
\int_\Omega \nabla w_i \cdot \nabla w_j \dx + \int_\Gamma K^{-1} w_i \vert_\Gamma \, w_j \vert_\Gamma \dH & = \lambda_i \delta_{ij},
\end{aligned}
\end{equation}
such that
\begin{align*}
\inner{A_{\Omega} \bm{w}_i}{\bm{p}}_{(\tilde V)',\tilde V} = \lambda_i (\bm{w}_i, \bm{p})_{\tilde{\bH}} := \lambda_i \Big ( \int_\Omega w_i p \dx + \int_\Gamma K^{-1} w_i \vert_\Gamma p \vert_\Gamma \dH \Big ) \quad \forall \bm{p} \in \tilde V.
\end{align*}
Equivalently, for any $i \in \N$, it holds that
\begin{align*}
- \Lap w_i = \lambda_i w_i \text{ in } \Omega, \quad \pdnu w_i + K^{-1} w_i = \lambda_i K^{-1} w_i \text{ on } \Gamma.
\end{align*}
For $m \in \N$, we introduce the finite-dimensional subspaces
\begin{align*}
W_m := \mathrm{span}\{w_1, \dots, w_m\}, \quad Y_m := \mathrm{span}\{y_1, \dots, y_m\},
\end{align*}
with the associated orthogonal projection $P_m$ in $\bH$ onto $W_m \times Y_m$.  For $\bm{u} = (u_1,u_2) \in \bH$, we use the notation $P_m u_1$ and $P_m u_2$ to denote the first and second components of $P_m \bm{u}$, respectively.  Then, for $\bm{v} = (v_1, v_2) \in \tilde V \times H^1(\Gamma)$, consider the operator $\bm{A} : \tilde V \times H^1(\Gamma) \to (\tilde{V})' \times H^1(\Gamma)'$ defined as
\begin{align}\label{op:A}
\inner{\bm{A}\bm{u}}{\bm{v}} := \int_\Omega \nabla u_1 \cdot \nabla v_1 \dx + \int_\Gamma K^{-1} u_1 \vert_\Gamma \, v_1 \vert_\Gamma \dH + \int_\Gamma \surf u_2 \cdot \surf v_2 + u_2 v_2 \dH.
\end{align}
By the orthonormality of $\{w_i\}_{i \in \N}$ in $\bH$ and $\{y_i\}_{i \in \N}$ in $L^2(\Gamma)$, as well as the property \eqref{ortho}, we have the following result.

\begin{lem}
For any $\bm{u} = (u_1, u_2) \in \bH$, it holds that
\begin{align*}
\inner{\bm{A} P_m \bm{u}}{P_m \bm{u}} & \geq \min(1,K^{-1}) \min(\lambda_1, \mu_1) \norm{P_m \bm{u}}_{\bH}^2,\\
 \inner{\bm{A}(\bm{u} - P_m \bm{u})}{\bm{u} - P_m \bm{u}} &\geq \min(1,K^{-1})\min(\lambda_{m},\mu_{m}) \norm{\bm{u} - P_m \bm{u}}_{\bH}^2.
\end{align*}
\end{lem}
\begin{proof}
Denoting by $\{u^1_j\}_{1 \leq j \leq m}$ and $\{u^2_j\}_{1 \leq j \leq m}$ the coefficients such that
\begin{align*}
P_m u_1 = \sum_{j=1}^{m} u^1_j w_j, \quad P_m u_2 = \sum_{j=1}^m u^2_j y_j,
\end{align*}
then, after integrating by parts, we obtain
\begin{align*}
& \inner{\bm{A} P_m \bm{u}}{P_m \bm{u}}\\
& \quad = \int_\Omega \Big ( \sum_{j=1}^m - u^1_j \Lap w_j  \Big ) \Big ( \sum_{i=1}^m u^1_i w_i \Big ) \dx + \int_\Gamma \Big ( \sum_{j=1}^m u^1_j (\pdnu w_j + K^{-1} w_j) \Big ) \Big ( \sum_{i=1}^m u^1_i w_i \Big ) \dH \\
& \qquad + \int_\Gamma \Big ( \sum_{j=1}^m u^2_j ( - \LB y_j + y_j ) \Big ) \Big ( \sum_{i=1}^m u^2_i y_i \Big ) \dH \\
& \quad = \int_\Omega \sum_{j=1}^m \lambda_j \abs{u^1_j}^2 \abs{w_j}^2 \dx + \int_\Gamma \sum_{j=1}^m \frac{\lambda_j }{K} \abs{u^1_j}^2 \abs{w_j \vert_\Gamma}^2 \dH + \int_\Gamma \sum_{j=1}^m \mu_j \abs{u^2_j}^2 \abs{y_j}^2 \dH \\
& \quad \geq \min(1,K^{-1}) \min(\lambda_1, \mu_1) \Big ( \norm{P_m u_1}_{L^2(\Omega)}^2 + \norm{P_m u_2}_{L^2(\Gamma)}^2 \Big )\\
&\quad = \min(1,K^{-1}) \min(\lambda_1, \mu_1) \norm{P_m\bm{u}}_{\bH}^2,
\end{align*}
once we employ the ordering of the eigenvalues $\{\lambda_j\}_{j \in \N}$ and $\{ \mu_j\}_{j \in \N}$.  The assertion for $\inner{\bm{A}(\bm{u} - P_m \bm{u})}{\bm{u} - P_m \bm{u}}$ is proved similarly with the observations $\lambda_m \leq \lambda_{m+1}$, $\mu_m \leq \mu_{m+1}$, and
\begin{align*}
\bm{u} - P_m \bm{u} = \Big ( \sum_{j=m+1}^{\infty} u_j^1 w_j, \sum_{j=m+1}^{\infty} u_j^2 y_j \Big ),
\end{align*}
and so we omit the details.
\end{proof}

By the generalized Poincar\'{e} inequality, it follows that
\begin{align*}
\inner{\bm{A} \bm{u}}{\bm{u}} \geq c_P \norm{u_1}_{H^1(\Omega)}^2 + \norm{u_2}_{H^1(\Gamma)}^2 \geq  c \norm{\bm{u}}_{\bV}^2
\end{align*}
for some positive constants $c_P, c$ depending only on $\Omega$ and $K$.  Let
\begin{align*}
\theta_{m} := \min(1,K^{-1}) \min(\lambda_{m}, \mu_{m}),
\end{align*}
then it holds that
\begin{align*}
\inner{\bm{A} \bm{u}}{\bm{u}}& = \frac{1}{2} \inner{\bm{A} \bm{u}}{\bm{u}} + \frac{1}{2} \inner{\bm{A} \bm{u}}{\bm{u}} \\
 & \geq \frac{c}{2} \norm{\bm{u}}_{\bV}^2 + \frac{1}{2} \Big ( \inner{\bm{A} P_m \bm{u}}{P_m \bm{u}} + \inner{\bm{A} (\bm{u} - P_m \bm{u})}{(\bm{u} - P_m \bm{u})} \Big ) \\
& \geq \frac{c}{2} \norm{\bm{u}}_{\bV}^2 + \frac{\theta_{m}}{2} \norm{\bm{u} - P_m \bm{u}}_{\bH}^2.
\end{align*}
Therefore, we arrive at
\begin{align*}
\inner{(\bm{A} + \theta_{m} P_m) \bm{u}}{\bm{u}} & \geq \frac{c}{2} \norm{\bm{u}}_{\bV}^2 + \frac{\theta_{m}}{2} \norm{\bm{u} - P_m \bm{u}}_{\bH}^2 + \theta_{m} \norm{P_m \bm{u}}_{\bH}^2 \\
& \geq \frac{c}{2} \norm{\bm{u}}_{\bV}^2 + \frac{\theta_m}{4} \norm{\bm{u}}_{\bH}^2.
\end{align*}

For fixed $(v,\psi)$ in $\bV$, and arbitrary $\bm{g} = (g_1, g_2), \bm{k} = (k_1, k_2) \in \bV$, we consider the following linearized operator $L((v,\psi)) =: L_{(v,\psi)} : \bV \to \bV'$ defined as
\begin{align}\label{defn:Lvpsi}
\notag & \inner{L_{(v,\psi)} \bm{g}}{\bm{k}}_{\bV',\bV} \\
\notag & \quad := \int_{\Omega} \nabla g_1 \cdot \nabla k_1 + f'(v + u_*) g_1 k_1 \dx + \int_{\Gamma} \surf g_2 \cdot \surf k_2 + f_{\Gamma}'(\psi + \phi_*) g_2 k_2 \dH \\
\notag & \qquad + \int_{\Gamma} K^{-1} (g_1 - h'(\psi + \phi_*) g_2)(k_1 - h'(\psi + \phi_*) k_2) \dH \\
& \qquad +  \int_\Gamma K^{-1} h''(\psi + \phi_*) (h(\psi + \phi_*) - (v + u_*))g_2 k_2  \dH.
\end{align}
Due to assumptions \eqref{ass:dom}--\eqref{ass:F}, we see from \eqref{defn:Lvpsi} that $L((v,\psi))$ is well-defined.
Besides, one observes that the domain of $L_{(v,\psi)}$ is $\bW$ and it is clear that $L_{(v,\psi)}$ is self-adjoint.
Associated to $L_{(0,0)} = L((0,0))$ is the bilinear form
\begin{align*}
 b(\bm{g}, \bm{k}) & := \int_{\Omega} \nabla g_1 \cdot \nabla k_1 + f'(u_*) g_1 k_1 \dx + \int_{\Gamma} \surf g_2 \cdot \surf k_2 + f_{\Gamma}'(\phi_*) g_2 k_2 \dH \\
 & \quad + \int_\Gamma K^{-1} (g_1 - h'(\phi_*) g_2)(k_1 - h'(\phi_*)k_2) + K^{-1} h''(\phi_*) (h(\phi_*) - u_*) g_2 k_2  \dH,
\end{align*}
which satisfies
\begin{align*}
b(\bm{g}, \bm{g}) & \geq \norm{\nabla g_1}_{L^2(\Omega)}^2  - \norm{f'(u_*)}_{L^{\infty}(\Omega)} \norm{g_1}_{L^2(\Omega)}^2 + \norm{\surf g_2}_{L^2(\Gamma)}^2 \\
& \quad - \norm{f_\Gamma'(\phi_*)}_{L^{\infty}(\Gamma)} \norm{g_2}_{L^2(\Gamma)}^2 + K^{-1} \norm{g_1 - h'(\phi_*)g_2}_{L^2(\Gamma)}^2 \\
& \quad - K^{-1} \norm{h''(\phi_*) (h(\phi_*) - u_*)}_{L^{\infty}(\Gamma)} \norm{g_2}_{L^2(\Gamma)}^2.
\end{align*}
Using the inequality
\begin{align*}
\norm{g_1 - h'(\phi_*) g_2}_{L^2(\Gamma)}^2 \geq \frac{1}{2} \norm{g_1}_{L^2(\Gamma)}^2 - \norm{h'(\phi_*) g_2}_{L^2(\Gamma)}^2,
\end{align*}
and the definition of the operator $\bm{A}$ from \eqref{op:A}, we see that
\begin{align*}
b(\bm{g}, \bm{g}) & \geq \frac{1}{2}\inner{\bm{A} \bm{g}}{\bm{g}} -  \norm{f'(u_*)}_{L^{\infty}(\Omega)} \norm{g_1}_{L^2(\Omega)}^2 - \Big ( \frac{1}{2} + K^{-1} \norm{h'(\phi_*)}_{L^{\infty}(\Gamma)}^2 \Big ) \norm{g_2}_{L^2(\Gamma)}^2 \\
& \quad - \Big ( \norm{f_\Gamma'(\phi_*)}_{L^{\infty}(\Gamma)} + K^{-1} \norm{h''(\phi_*) (h(\phi_* ) - u_*)}_{L^{\infty}(\Gamma)} \Big )\norm{g_2}_{L^2(\Gamma)}^2 \\
& \geq \frac{c}{4} \norm{\bm{g}}_{\bV}^2 + \frac{\theta_m}{4} \norm{\bm{g} - P_m \bm{g}}_{\bH}^2 - c_* \norm{\bm{g}}_{\bH}^2,
\end{align*}
for some positive constant $c_*$ depending only on the $L^{\infty}(\Gamma)$ norm of $h'(\phi_*)$, $h''(\phi_*)$, $h(\phi_*) - u_*$ and $f_\Gamma'(\phi_*)$.
Since the eigenvalues satisfy $\lambda_j \to \infty$, $\mu_j \to \infty$ as $j \to \infty$, we can choose $m$ sufficiently large so that
\begin{align}\label{defn:thetam}
\theta_m & = \min(1,K^{-1}) \min(\lambda_{m}, \mu_{m}) > 8 c_*.
\end{align}
Then we can prove the following result:
\begin{lem}\label{lem:L00}
Fix $m \in \N$ such that \eqref{defn:thetam} is valid.  For any $\bm{w} = (w_1, w_2) \in \bH$, there exists a unique solution $\bm{g} = (g_1, g_2) \in \bW$ to the abstract equation
\begin{align}\label{absequ}
\LL_{(0,0)} \bm{g}:= (L_{(0,0)} + \theta_m P_m) \bm{g} = \bm{w}.
\end{align}
Furthermore, it holds that
\begin{align*}
\norm{\bm{g}}_{\bW} \leq C \norm{\bm{w}}_{\bH}.
\end{align*}
\end{lem}
\begin{proof}
Thanks to \eqref{defn:thetam}, we can deduce that
\begin{equation}\label{coer}
\begin{aligned}
&\inner{(L_{(0,0)} + \theta_{m} P_m) \bm{g}}{\bm{g}}_{\bV',\bV}\\
&\quad = b(\bm{g}, \bm{g}) + \theta_m \inner{ P_m \bm{g}}{\bm{g}}_{\bV',\bV} \\
&\quad  \geq \frac{c}{4} \norm{\bm{g}}_{\bV}^2 + \frac{\theta_m}{4} \norm{\bm{g} - P_m \bm{g}}_{\bH}^2 + \frac{\theta_m}{4} \norm{P_m \bm{g}}_{\bH}^2 - c_* \norm{\bm{g}}_{\bH}^2 \\
&\quad \geq \frac{c}{4} \norm{\bm{g}}_{\bV}^2 + \Big ( \frac{\theta_m}{8} - c_* \Big ) \norm{\bm{g}}_{\bH}^2\\
&\quad \geq \frac{c}{4} \norm{\bm{g}}_{\bV}^2.
\end{aligned}
\end{equation}
From \eqref{coer}, the operator $\LL_{(0,0)}$ is coercive on $\bV$.  Furthermore, it is clear that $\LL_{(0,0)}$ is bounded on $\bV$, and so the unique solvability of \eqref{absequ} follows directly from the Lax--Milgram theorem. Moreover, from the coercivity of $\LL_{(0,0)}$, we obtain
\begin{align*}
\frac{c}{4}\norm{\bm{g}}_{\bV}^2 \leq \inner{\LL_{(0,0)} \bm{g}}{\bm{g}}  = (\bm{w}, \bm{g})_{\bH} \leq \norm{\bm{w}}_{\bH} \norm{\bm{g}}_{\bH},
\end{align*}
leading to the $\bV$-stability estimate
\begin{align*}
\norm{\bm{g}}_{\bV} \leq C \norm{\bm{w}}_{\bH}.
\end{align*}
For regularity in $\bW$, we observe that $\bm{g} = (g_1, g_2)$ is a weak solution to the linear system
\begin{subequations}
\begin{alignat}{3}
- \Lap g_1 + f'(u_*) g_1 + \theta_m P_m g_1 = w_1 &\quad  \text{ in } \Omega, \label{lin:1} \\
\pdnu g_1 + K^{-1} g_1 = K^{-1} h'(\phi_*) g_2 &\quad  \text{ on } \Gamma, \label{lin:2} \\
\notag - \LB g_2 + f_\Gamma'(\phi_*) g_2 + \theta_m P_m g_2 + K^{-1} (h'(\phi_*))^2 g_2 & \\
 - K^{-1}  h'(\phi_*) g_1 + K^{-1} h''(\phi_*)(h(\phi_*) - u_*) g_2 = 0 &\quad  \text{ on } \Gamma. \label{lin:3}
\end{alignat}
\end{subequations}
Applying elliptic regularity for the third equation \eqref{lin:3} yields
\begin{align*}
\norm{g_2}_{H^2(\Gamma)} & \leq C \norm{g_2}_{H^1(\Gamma)} + C \Big(\norm{f_\Gamma'(\phi_*)}_{L^{\infty}(\Gamma)} + \theta_m + K^{-1} \norm{h'(\phi_*)}_{L^{\infty}(\Gamma)}^2 \Big )\norm{g_2}_{L^2(\Gamma)} \\
& \quad +  C\norm{h'(\phi_*)}_{L^{\infty}(\Gamma)} \norm{g_1}_{L^2(\Gamma)} + C \norm{h''(\phi_*)(h(\phi_*) - u_*)}_{L^{\infty}(\Gamma)} \norm{g_2}_{L^2(\Gamma)} \\
& \leq C \norm{\bm{g}}_{\bV} \leq C \norm{\bm{w}}_{\bH}.
\end{align*}
Then, the regularity for $\phi_*$ and $g_2$ imply that $h'(\phi_*) g_2 \in H^{\frac{1}{2}}(\Gamma)$, and so by elliptic regularity for the system \eqref{lin:1}--\eqref{lin:2}, we obtain
\begin{align*}
\norm{g_1}_{H^2(\Omega)} & \leq C \norm{g_1}_{H^1(\Omega)} + C \Big ( \norm{f'(u_*)}_{L^{\infty}(\Omega)} + \theta_m \Big )\norm{g_1}_{L^2(\Omega)} + C \norm{h'(\phi_*)g_2}_{H^{\frac{1}{2}}(\Gamma)} \\
& \leq C \norm{g_1}_{H^1(\Omega)} + C \norm{\phi_*}_{H^2(\Gamma)} \norm{g_2}_{H^2(\Gamma)} \leq C \norm{\bm{w}}_{\bH}.
\end{align*}
The proof is complete.
\end{proof}

\paragraph{Step 2. Analysis of a certain nonlinear operator.}
For sufficiently large $m$ chosen in Step 1, we set $\Pi_m := \theta_m P_m$.
For any $(v,\psi), (w,\xi) \in \bV$, consider the nonlinear operator $\mathcal{N}: \bV \to \bV'$ defined as
\begin{align}\label{defn:NN}
\inner{\mathcal{N}(v,\psi)}{(w,\xi)}_{\bV',\bV} = (\Pi_m (v,\psi), (w,\xi))_{\bH} + \inner{\M(v,\psi)}{(w,\xi)}_{\bV}.
\end{align}
Besides, for given $(v,\psi) \in \bV$, we define the linear operator $\LL_{(v,\psi)} : \bV \to \bV'$ as
\begin{align}\label{defn:LL}
\LL_{(v,\psi)} := \Pi_m + L_{(v,\psi)} = \theta_m P_m + L_{(v,\psi)},
\end{align}
where $L_{(v,\psi)}$ is given in \eqref{defn:Lvpsi}. Then we have

\begin{lem}
For any $(v,\psi) \in \bV$, the operator $\mathcal{N}$ is Fr\'{e}chet differentiable with derivative $\der \mathcal{N}(v,\psi) = \LL_{(v,\psi)}$, i.e.,
\begin{align*}
\frac{\norm{\mathcal{N}(v+g_1, \psi + g_2) - \mathcal{N}(v, \psi) - \LL_{(v,\psi)}(g_1, g_2)}_{\bV'}}{\norm{(g_1, g_2)}_{\bV}} \to 0 \text{ as } \norm{(g_1, g_2)}_{\bV} \to 0.
\end{align*}
\end{lem}

\begin{proof}
For arbitrary $(g_1, g_2), (k_1, k_2) \in \bV$, we compute that
\begin{align*}
& \inner{\mathcal{N}(v+ g_1, \psi +  g_2) - \mathcal{N}(v,\psi) - \mathcal{L}_{(v,\psi)}(g_1, g_2)}{(k_1, k_2)}_{\bV',\bV} \\
& \quad = \inner{\M(v+ g_1, \psi + g_2) - \M(v,\psi) - L_{(v,\psi)}(g_1, g_2)}{(k_1, k_2)}_{\bV',\bV} \\
& \quad = \int_\Omega [f(u+  g_1) - f(u) -  f'(u) g_1] k_1 \dx  + \int_\Gamma [f_\Gamma(\phi + g_2) - f_\Gamma (\phi) - f_\Gamma'(\phi) g_2] k_2 \dH \\
& \qquad - \int_\Gamma K^{-1} [ h(\phi + g_2) - h(\phi) - h'(\phi) g_2] k_1 + K^{-1} u k_2  [ h'(\phi + g_2) - h'(\phi) + h''(\phi) g_2] \dH \\
& \qquad + \int_\Gamma K^{-1} h'(\phi + g_2) k_2 [ h(\phi+ g_2) - h(\phi) - h'(\phi) g_2] \dH \\
& \qquad + \int_\Gamma K^{-1}  h(\phi) k_2 [h'(\phi+ g_2) - h'(\phi) - h''(\phi) g_2] \dH \\
& \qquad + \int_\Gamma K^{-1} g_2 k_2 [h'(\phi) h'(\phi+ g_2) - (h'(\phi))^2 ]-  K^{-1} g_1 k_2 [h'(\phi + g_2) - h'(\phi)] \dH.
\end{align*}
By the Newton--Leibniz formula
\begin{align*}
f(u+ g_1) - f(u) - f'(u) g_1 = \int_0^1 \int_0^1 f''(sz (u + g_1) + (1-sz) u) g_1^2 \ds \dz,
\end{align*}
and the growth assumption \eqref{ass:F} for $f''$, we obtain
\begin{align*}
& \abs{\int_\Omega [f(u+g_1) - f(u) -f'(u) g_1] k_1 \dx} \\
& \quad \leq \int_0^1 \int_0^1 \norm{f''(sz(u+g_1) + (1-sz) u)}_{L^2(\Omega)} \ds \dz \norm{g_1}_{L^6(\Omega)}^2 \norm{k_1}_{L^6(\Omega)} \\
& \quad \leq C \Big ( 1 + \norm{u+g_1}_{L^6(\Omega)}^{p} + \norm{u}_{L^6(\Omega)}^{p} \Big ) \norm{g_1}_{H^1(\Omega)}^2 \norm{k_1}_{H^1(\Omega)}.
\end{align*}
Arguing similarly for the other terms with assumptions \eqref{ass:h} and \eqref{ass:F} in mind, we have
\begin{align*}
& \abs{\inner{\mathcal{N}(v+g_1, \psi+g_2) - \mathcal{N}(v,\psi) - \mathcal{L}_{(v,\psi)}(g_1,g_2)}{(k_1,k_2)}_{\bV',\bV}} \\
& \quad \leq C  \Big ( 1 + \norm{g_1}_{H^1(\Omega)}^{p} + \norm{u}_{H^1(\Omega)}^{p} \Big ) \norm{g_1}_{H^1(\Omega)}^2 \norm{k_1}_{H^1(\Omega)} \\
& \qquad + C\Big ( 1 + \norm{g_2}_{H^1(\Gamma)}^{q} + \norm{\phi}_{H^1(\Gamma)}^{q} \Big ) \norm{g_2}_{H^1(\Gamma)}^2 \norm{k_2}_{H^1(\Gamma)} \\
& \qquad + C\norm{h''}_{L^{\infty}(\R)} \norm{g_2}_{H^1(\Gamma)}^2 \norm{k_1}_{L^2(\Gamma)} \\
& \qquad + C\Big ( 1 + \norm{g_2}_{H^1(\Gamma)}^{q} + \norm{\phi}_{H^1(\Gamma)}^{q} \Big )\norm{u}_{L^2(\Gamma)} \norm{g_2}_{H^1(\Gamma)}^2 \norm{k_2}_{H^1(\Gamma)} \\
& \qquad + C\norm{h' h''}_{L^{\infty}(\R)} \norm{g_2}_{H^1(\Gamma)}^2 \norm{k_2}_{H^1(\Gamma)}^2 \\
& \qquad + C\Big ( 1 + \norm{g_2}_{H^1(\Gamma)}^{q} + \norm{\phi}_{H^1(\Gamma)}^{q} \Big ) \norm{h(\phi)}_{L^2(\Gamma)} \norm{g_2}_{H^1(\Gamma)}^2 \norm{k_2}_{H^1(\Gamma)} \\
& \qquad + C\norm{h' h''}_{L^{\infty}(\R)} \norm{g_2}_{H^1(\Gamma)}^2 \norm{k_2}_{H^1(\Gamma)} + C\norm{h''}_{L^{\infty}(\R)} \norm{g_1}_{L^2(\Gamma)} \norm{g_2}_{H^1(\Gamma)} \norm{k_2}_{H^1(\Gamma)} \\
& \quad \leq C \Big ( 1 + \norm{g_1}_{H^1(\Omega)}^{p} + \norm{u}_{H^1(\Omega)}^{\max(p,2)} + \norm{g_2}_{H^1(\Gamma)}^{2q} + \norm{\phi}_{H^1(\Gamma)}^{2q} \Big ) \norm{\bm{g}}_{\bV}^2 \norm{\bm{k}}_{\bV},
\end{align*}
and hence
\begin{align*}
& \frac{\norm{\mathcal{N}(v+g_1, \psi + g_2) - \mathcal{N}(v,\psi) - \mathcal{L}_{(v,\psi)}(g_1, g_2)}_{\bV'}}{\norm{\bm{g}}_{\bV}}\\
& \quad \leq C \Big ( 1 + \norm{g_1}_{H^1(\Omega)}^{p} + \norm{u}_{H^1(\Omega)}^{\max(p,2)} + \norm{g_2}_{H^1(\Gamma)}^{2q} + \norm{\phi}_{H^1(\Gamma)}^{2q} \Big ) \norm{\bm{g}}_{\bV} \to 0 \text{ as } \norm{\bm{g}}_{\bV} \to 0,
\end{align*}
which implies the desired assertion.
\end{proof}

We can deduce from the analyticity of $F$, $F_\Gamma$ and $h$ that the mappings
\begin{align*}
L^\infty(\Omega) \ni u & \mapsto f(u) \in L^\infty(\Omega), \\
L^\infty(\Gamma) \ni \phi & \mapsto f_\Gamma(\phi) \in L^\infty(\Gamma), \\
L^\infty(\Gamma) \ni \phi & \mapsto h(\phi) \in L^\infty(\Gamma)
\end{align*}
are analytic (in the sense of \cite[Definition 2.4]{HJ}).  Then, by the embedding $\bW \subset L^{\infty}(\Omega) \times L^{\infty}(\Gamma)$, it follows that the restricted operator $\mathcal{N}_{W} := \mathcal{N} \vert_{\bW} : \bW \to \bH$ is also analytic.  Furthermore, since $\der \mathcal{N}_{W}((0,0)) = \LL_{(0,0)}$ is a bijection by Lemma \ref{lem:L00}, we can invoke the analytic implicit function theorem (see for example \cite[Corollary 4.37, p.~172]{Zeidler}) to deduce the existence of small neighbourhoods around the origins, $\bm{U}_1(0) \subset \bW$ and $\bm{U}_2(0) \subset \bH$, as well as an analytic and bijective inverse $\Psi := \mathcal{N}_{W}^{-1}: \bm{U}_2(0) \to \bm{U}_1(0)$ such that
\begin{align}
\mathcal{N}_W(\Psi(\bm{g})) = \bm{g} \text{ for all } \bm{g} \in \bm{U}_2(0), \quad \Psi(\mathcal{N}_{W}(\bm{h})) = \bm{h} \text{ for all } \bm{h} \in \bm{U}_1(0)
\end{align}
and
\begin{subequations}\label{Lip:N:str}
\begin{alignat}{2}
\norm{\Psi(\bm{g}_1) - \Psi(\bm{g}_2)}_{\bW} & \leq C \norm{\bm{g}_1 - \bm{g}_{2}}_{\bH}\quad  \text{ for all } \bm{g}_1, \bm{g}_2 \in \bm{U}_2(0), \\
\norm{\mathcal{N}_{W}(\bm{h}_1) - \mathcal{N}_{W}(\bm{h}_2)}_{\bH} & \leq C \norm{\bm{h}_1 - \bm{h}_2}_{\bW}\quad  \text{ for all } \bm{h}_1, \bm{h}_2 \in \bm{U}_1(0).
\end{alignat}
\end{subequations}
On the other hand, since $\der \mathcal{N}((0,0)) = \LL_{(0,0)}$ is a bijection, by the classical local inversion theorem (see for example \cite[Theorem 4.F, p.~172]{Zeidler}), the operator $\mathcal{N}: \bV \to \bV'$ is a $C^1$-diffeomorphism near $(0,0)$.  This assures the existence of neighbourhoods $\widehat{\bm{U}}_1(0) \subset \bV$ and $\widehat{\bm{U}}_2(0) \subset \bV'$ such that
\begin{subequations}
\begin{alignat}{2}
\norm{\mathcal{N}^{-1}(\bm{g}_1) - \mathcal{N}^{-1}(\bm{g}_2)}_{\bV} & \leq C \norm{\bm{g}_1 - \bm{g}_2}_{\bV'}\quad \text{ for all } \bm{g}_1, \bm{g}_2 \in \widehat{\bm{U}}_2(0), \label{Lip:Nin} \\
\norm{\mathcal{N}(\bm{h}_1) - \mathcal{N}(\bm{h}_2)}_{\bV'} & \leq C \norm{\bm{h}_1 - \bm{h}_2}_{\bV}\quad \text{ for all } \bm{h}_1, \bm{h}_2 \in \widehat{\bm{U}}_1(0). \label{Lip:N}
\end{alignat}
\end{subequations}
In particular, in the intersection $\widehat{\bm{U}}_2(0) \cap \bm{U}_2(0)$ we have the identification $\Psi = \mathcal{N}^{-1}$.

\paragraph{Step 3. Derivation of the {\L}ojasiewicz--Simon inequality.}
We now define a function $\J: \R^m \times \R^m \to \R$ by
\begin{align}\label{defn:JJ}
\J(\bm{\xi}, \bm{\zeta}) = \E \Big ( \Psi \Big (\sum_{j=1}^m \xi_j w_j,\, \sum_{j=1}^m \zeta_j  y_j \Big ) \Big ),
\end{align}
where $\{ w_j \}_{j \in \N}$ and $\{y_j \}_{j \in \N}$ are the basis functions introduced in Step 1, and $m$ is the index such that \eqref{coer} holds.  For $\abs{\bm{\xi}}$ and $\abs{\bm{\zeta}}$ sufficiently small, it holds that
\begin{align*}
\Big (\sum_{j=1}^m \xi_j w_j, \sum_{j=1}^m \zeta_j y_j \Big ) \in \widehat{\bm{U}}_2(0) \cap \bm{U}_2(0),
\end{align*}
over which the mapping $\Psi$ is analytic. Together with the analyticity of $\E$, we infer that $\J$ is analytic with respect to $\bm{\xi}$ and $\bm{\zeta}$.  Then, applying the classical {\L}ojasiewicz inequality (see for instance \cite[Proposition 2.3]{HJ}) there exists $\sigma > 0$ and $0 < \mu \leq  \frac{1}{2}$ such that for all  $(\bm{\xi}, \bm{\zeta}) \in \R^{m} \times \R^m$ with $|(\bm{\xi}, \bm{\zeta})|_{\R^{2m}} < \sigma$, it holds
\begin{align}\label{Loj}
\abs{\nabla \J(\bm{\xi}, \bm{\zeta})}_{\R^{2m}} \geq  \abs{\J(\bm{\xi}, \bm{\zeta}) - \J(\bm{0}, \bm{0})}^{1-\mu}.
\end{align}

Next, we consider perturbations $(v,\psi)$ satisfying
\begin{align*}
(v,\psi) \in \widehat{\bm{U}}_1(0) \subset \bV, \quad \Pi_m (v,\psi) = \Big (\sum_{j=1}^m \xi_j w_j, \sum_{j=1}^m \zeta_j y_j  \Big ) \in \widehat{\bm{U}}_2(0) \cap \bm{U}_2(0)
\end{align*}
for some vectors $\bm{\xi}, \bm{\zeta} \in \R^m$, where we recall $\Pi_m = \theta_m P_m$ with $\theta_m$ defined in \eqref{defn:thetam} and $P_m$ is the orthogonal projection from $\bH$ to the product finite-dimensional subspace $W_m \times Y_m$ spanned by the first $m$ eigenfunctions in $\{w_i\}_{i \in \N}$ and $\{y_i\}_{i \in \N}$.  Then, from the relation $\mathcal{N}( \Psi(\Pi_m(v,\psi))) = \Pi_m(v,\psi)$ we obtain that
\begin{align*}
\der \Psi(\Pi_m(v,\psi)) = (\der \mathcal{N}( \Psi(\Pi_m(v,\psi))))^{-1} = \LL_{\Psi(\Pi_m(v,\psi))}^{-1},
\end{align*}
once we recall the Fr\'{e}chet derivative of $\mathcal{N}$ is the operator $\LL_{(v,\psi)}$ defined in \eqref{defn:LL}.
Hence, for the gradient appearing on the left-hand side of the {\L}ojasiewicz inequality \eqref{Loj}, with a short calculation we obtain
\begin{subequations}\label{pdJ}
\begin{alignat}{2}
\frac{\pd \J(\bm{\xi}, \bm{\zeta})}{\pd \xi_i} & = \inner{\M(\Psi(\Pi_m(v,\psi)))}{[\der \Psi(\Pi_m(v,\psi))](w_i,0)}_{\bV',\bV}, \label{pdJb} \\
\frac{\pd \J(\bm{\xi}, \bm{\zeta})}{\pd \zeta_i} & = \inner{\M(\Psi(\Pi_m(v,\psi)))}{[\der \Psi(\Pi_m(v,\psi))](0, y_i)}_{\bV',\bV}, \label{pdJs}
\end{alignat}
\end{subequations}
where for $(b_1,b_2) \in \bV'$, the pair $[\der \Psi(\Pi_m(v,\psi))](b_1,b_2) = \LL_{\Psi(\Pi_m(v,\psi))}^{-1}(b_1,b_2) = (g_1,g_2) \in \bV$ satisfies
\begin{align*}
& \int_\Omega \nabla g_1 \cdot \nabla k_1 + f'(\Psi_1 + u_*) g_1 k_1 \dx + \int_\Gamma \surf g_2 \cdot \surf k_2 + f_\Gamma'(\Psi_2  + \phi_*) g_2 k_2 \dH \\
& \quad + \int_\Gamma K^{-1}(g_1 - h'(\Psi_2 + \phi_*) g_2)(k_1 - h'(\Psi_2 + \phi_*) k_2) \dH \\
& \quad + \int_\Gamma K^{-1} h''(\Psi_2 + \phi_*)( h(\Psi_2 + \phi_*) - (\Psi_1 + u_*)) g_2 k_2 \dH \\
& = \inner{(b_1, b_2)}{(k_1, k_2)}_{\bV',\bV} \quad \text{ for all } (k_1, k_2) \in \bV,
\end{align*}
and $(\Psi_1, \Psi_2)$ are the bulk and surface components of $\Psi(\Pi_m(v,\psi))$, respectively.  In particular, $[\der \Psi(\Pi_m(v,\psi))](w_i, 0)$ and $[\der \Psi(\Pi_m(v,\psi))](0,y_i)$ both belong to $\bV$, and so the right-hand sides of \eqref{pdJb} and \eqref{pdJs} are well-defined.  From the analyticity of $\Psi$, and the fact that $(w_i, y_i) \in \bW$, we have
\begin{align*}
\norm{[\der \Psi(\Pi_m(v,\psi))](w_i, 0)}_{\bV} \leq C, \quad \norm{[\der \Psi(\Pi_m(v,\psi))](0,y_i)}_{\bV} \leq C,
\end{align*}
and so from \eqref{pdJb} and \eqref{pdJs} we obtain
\begin{equation}\label{grad1}
\begin{aligned}
\abs{\nabla \J(\bm{\xi}, \bm{\zeta})}_{\R^{2m}} & \leq C \norm{\M(\Psi(\Pi_m(v,\psi)))}_{\bV'} \\
& \leq C \norm{\M(\Psi(\Pi_m(v,\psi))) - \M(v,\psi)}_{\bV'} + C\norm{\M(v,\psi)}_{\bV'}.
\end{aligned}
\end{equation}
Recalling the definition \eqref{defn:NN}, we infer that $(v,\psi) = \mathcal{N}^{-1}(\M(v,\psi) + \Pi_m(v,\psi))$.  Then, by invoking the estimates \eqref{Lip:Nin} and \eqref{Lip:N}, we arrive at
\begin{equation}\label{grad2}
\begin{aligned}
& \norm{\M(\Psi(\Pi_m(v,\psi))) - \M(v,\psi)}_{\bV'} \\
& \quad \leq \norm{\mathcal{N}(\Psi(\Pi_m(v,\psi))) - \mathcal{N}(v,\psi)}_{\bV'} + \norm{\Pi_m(\Psi(\Pi_m(v,\psi)) - (v,\psi))}_{\bV'} \\
& \quad \leq C \norm{\Psi(\Pi_m(v,\psi)) - (v,\psi)}_{\bV} \\
& \quad = C \norm{\mathcal{N}^{-1}(\Pi_m(v,\psi)) - \mathcal{N}^{-1}(\M(v,\psi) + \Pi_m(v,\psi))}_{\bV} \\
& \quad \leq C \norm{\M(v,\psi)}_{\bV'}.
\end{aligned}
\end{equation}
Combining \eqref{grad1} and \eqref{grad2} leads to
\begin{align}\label{grad3}
\abs{\nabla \J(\bm{\xi}, \bm{\zeta})}_{\R^{2m}} \leq C \norm{\M(v,\psi)}_{\bV'}.
\end{align}
Meanwhile, for the left-hand side of \eqref{Loj} we observe that
\begin{align*}
\J(\bm{0}, \bm{0}) = \E(\Psi(\bm{0},\bm{0})) = \E(0,0) = E(u_*, \phi_*),
\end{align*}
and so from \eqref{Loj} and \eqref{grad3} we infer that
\begin{equation}\label{grad4}
\begin{aligned}
C \norm{\M(v,\psi)}_{\bV'} & \geq  \abs{\nabla \J(\bm{\xi}, \bm{\zeta})}_{\R^{2m}} \geq  \abs{\J(\bm{\xi}, \bm{\zeta}) - E(u_*, \phi_*)}^{1-\mu} \\
& = \abs{\J(\bm{\xi}, \bm{\zeta}) - \E(v,\psi) + \E(v,\psi) - \E(0, 0)}^{1-\mu} \\
& \geq \frac{1}{2} \abs{\E(v,\psi) - \E(0,0)}^{1-\mu} - C\abs{\J(\bm{\xi}, \bm{\zeta}) - \E(v, \psi)}^{1-\mu}.
\end{aligned}
\end{equation}
Hence, to obtain the desired inequality \eqref{LS} it suffices to control the second term on the right-hand side of \eqref{grad4}.
Employing the Newton--Leibniz formula, we have
\begin{align*}
& \abs{\J(\bm{\xi}, \bm{\zeta}) - \E(v,\psi)} = \abs{\E(\Psi(\Pi_m(v,\psi))) - \E(v,\psi)} \\
& \quad = \abs{\int_0^1 \frac{d}{dt} \E( (v,\psi) + t ( \Psi(\Pi_m(v,\psi)) - (v,\psi)) \dt} \\
& \quad = \abs{\int_0^1 \inner{\M((v,\psi)+ t (\Psi(\Pi_m(v,\psi)) - (v,\psi)))}{\Psi(\Pi_m(v,\psi)) - (v,\psi)}_{\bV} \dt} \\
& \quad \leq \norm{\Psi(\Pi_m(v,\psi)) - (v,\psi)}_{\bV} \int_0^1 \norm{\M((v,\psi) + t (\Psi(\Pi_m(v,\psi)) - (v,\psi)))}_{\bV'} \dt \\
& \quad \leq \norm{\Psi(\Pi_m(v,\psi)) - (v,\psi)}_{\bV}  \\
& \qquad \times  \int_0^1 \norm{\M((v,\psi) + t (\Psi(\Pi_m(v,\psi)) - (v,\psi))) - \M(v,\psi)}_{\bV'} + \norm{\M(v,\psi)}_{\bV'} \dt.
\end{align*}
Employing a similar argument to the derivation of \eqref{grad2}, we see that
\begin{align*}
\norm{\Psi(\Pi_m(v,\psi)) - (v,\psi)}_{\bV} & = \norm{\mathcal{N}^{-1}(\Pi_m(v,\psi)) - \mathcal{N}^{-1}(\M(v,\psi) + \Pi_m(v,\psi))}_{\bV} \\
& \leq C \norm{\M(v,\psi)}_{\bV'},
\end{align*}
and
\begin{align*}
& \norm{\M((v,\psi) + t (\Psi(\Pi_m(v,\psi)) - (v,\psi))) - \M(v,\psi)}_{\bV'} \\
& \quad \leq C t \norm{\Psi(\Pi_m(v,\psi)) - (v,\psi)}_{\bV} \leq C t \norm{\M(v,\psi)}_{\bV'}.
\end{align*}
Hence, we get
\begin{align}\label{grad5}
\abs{\J(\bm{\xi}, \bm{\zeta}) - \E(v,\psi)} \leq C \norm{\M(v,\psi)}_{\bV'}^2,
\end{align}
and from \eqref{grad4} this leads to
\begin{align*}
\abs{\E(v,\psi) - \E(0,0)}^{1-\mu} \leq C \norm{\M(v,\psi)}_{\bV'} \Big (1 + \norm{\M(v,\psi)}^{2(1-\mu)-1}_{\bV'} \Big ).
\end{align*}
Since $\mu \in (0,\frac{1}{2}]$ and $2(1-\mu)-1 \geq 0$, we can find a positive constant $\beta_0 < \sigma$ such that for $\norm{(v,\psi)}_{\bV} < \beta_0$,
\begin{align*}
\norm{\M(v,\psi)}^{2(1-\mu)-1}_{\bV'} \leq 1,
\end{align*}
which implies that
\begin{align*}
\frac{1}{2C}\abs{\E(v,\psi) - \E(0,0)}^{1-\mu} \leq \norm{\M(v,\psi)}_{\bV'} \text{ for all } \norm{(v,\psi)}_{\bV} < \beta_0.
\end{align*}
Let $\eps \in (0,\mu)$ be an exponent and $\beta < \beta_0$ be a positive constant such that
\begin{align*}
\frac{1}{2C} \abs{\E(v,\psi) - \E(0,0)}^{-\eps} \geq 1 \text{ for all } \norm{(v,\psi)}_{\bV} < \beta.
\end{align*}
Then, for $\theta := \mu - \eps \in (0,\frac{1}{2})$, we have
\begin{align*}
\abs{\E(v,\psi) - \E(0,0)}^{1-\theta} \leq \norm{\M(v,\psi)}_{\bV'} \text{ for all } \norm{(v,\psi)}_{\bV} < \beta,
\end{align*}
which is exactly \eqref{LS}.

The proof of Theorem \ref{thm:LSa} is complete.

\section{Long-time Behaviour}\label{sec:Long}

First, we deduce the following result on the decay of time derivatives $(\partial_t u, \partial_t \phi)$:
\begin{prop}\label{lem:timederiv}
Let $(u,\phi)$ be the global strong solution to problem \eqref{ACAC}. It holds that
\begin{align*}
\lim_{t \to + \infty} \norm{(\pd_t u(t), \pd_t \phi(t))}_{\bH} = 0.
\end{align*}
\end{prop}
\begin{proof}
We take the time derivative of \eqref{u} and test with $\pd_t u$, leading to
\begin{align*}
& \frac{1}{2} \frac{d}{dt} \norm{\pd_t u}_{L^2(\Omega)}^2 + \norm{\nabla \pd_t u}_{L^2(\Omega)}^2 + K^{-1} \norm{\pd_t u}_{L^2(\Gamma)}^2 \\
& \quad  = -\int_\Omega f'(u)\abs{\pd_t u}^2 \dx + \int_\Gamma K^{-1} \pd_t h(\phi) \pd_t u \dH \\
& \quad \leq c_4 \norm{\pd_t u}_{L^2(\Omega)}^2 + \frac{1}{2K} \norm{\pd_t u}_{L^2(\Gamma)}^2 + \frac{1}{2K}  \norm{\pd_t h(\phi)}_{L^2(\Gamma)}^2.
\end{align*}
Similarly, taking the time derivative of \eqref{phi} and testing with $\pd_t \phi$ leads to
\begin{align*}
& \frac{1}{2} \frac{d}{dt} \norm{\pd_t \phi}_{L^2(\Gamma)}^2 + \norm{\surf \pd_t \phi}_{L^2(\Gamma)}^2 + K^{-1} \norm{\pd_t h(\phi)}_{L^2(\Gamma)}^2 \\
& \quad = - \int_\Gamma f_\Gamma'(\phi) \abs{\pd_t \phi}^2 + h''(\phi) \abs{\pd_t \phi}^2 \pdnu u - K^{-1}\pd_t h(\phi) \pd_t u \dH \\
& \quad \leq c_4 \norm{\pd_t \phi}_{L^2(\Gamma)}^2 + \frac{1}{2K} \norm{\pd_t h(\phi)}_{L^2(\Gamma)}^2 + \frac{1}{2K} \norm{\pd_t u}_{L^2(\Gamma)}^2 \\
& \qquad + \norm{h''}_{L^{\infty}(\R)} \norm{\pdnu u}_{L^2(\Gamma)} \norm{\pd_t \phi}_{L^4(\Gamma)}^2.
\end{align*}
Adding these two inequalities yields
\begin{align*}
& \frac{1}{2} \frac{d}{dt} \norm{(\pd_t u, \pd_t \phi)}_{\bH}^2 + \norm{(\nabla \pd_t u, \surf \pd_t \phi)}_{\bH}^2 \\
& \quad \leq c_4 \norm{(\pd_t u, \pd_t \phi)}_{\bH}^2 + C \norm{h''}_{L^{\infty}(\R)} \norm{\pdnu u}_{L^2(\Gamma)} \Big ( \eps \norm{\surf \pd_t \phi}_{L^2(\Gamma)}^2 + C_\eps \norm{\pd_t \phi}_{L^2(\Gamma)}^2 \Big ),
\end{align*}
where we employed the Gagliardo--Nirenburg inequality \eqref{GN} and Young's inequality.  Choosing $\eps$ sufficiently small, we arrive at
\begin{align}\label{time:der}
\frac{d}{dt} \norm{(\pd_t u, \pd_t \phi)}_{\bH}^2 + \norm{(\nabla \pd_t u, \surf \pd_t \phi)}_{\bH}^2 &\leq C\norm{(\pd_t u, \pd_t \phi)}_{\bH}^2\nonumber\\
&\leq  \norm{(\pd_t u, \pd_t \phi)}_{\bH}^4 + C.
\end{align}
Invoking \cite[Lemma 6.2.1]{Zheng} and using the fact that $\norm{(\pd_t u, \pd_t \phi)}_{L^{2}(0,+\infty;\bH)} < \infty$ (recall \eqref{Est1}), we deduce the desired assertion.
\end{proof}

Next, thanks to Theorem \ref{thm:Exist}, for any initial data $(u_0,\phi_0) \in \bW$ satisfying \eqref{ass:ini}, the unique global solution $(u,\phi)$ to problem \eqref{ACAC} allows us to define the $\omega$-limit set $\omega(u_0, \phi_0)$ as
\begin{equation}\label{omega}
\begin{aligned}
\omega(u_0, \phi_0) = \Big \{ (U,\Phi) : \exists \{ t_k \}_{k \in \N}, \, t_k \nearrow +\infty \text{ s.t. } (u(t_k), \phi(t_k)) \to (U,\Phi) \text{ in } \bW \Big \}.
\end{aligned}
\end{equation}
Then by Theorem \ref{thm:Exist} and the Lyapunov structure \eqref{Lyap} of problem \eqref{ACAC}, it is standard to conclude the following result:
\begin{lem}\label{lem:omega}
Under assumptions \eqref{ass:dom}--\eqref{ass:F}, for any $(u_0, \phi_0) \in \bW$ satisfying \eqref{ass:ini}, the set $\omega(u_0,\phi_0)$ is a non-empty compact subset in $\bW$.
Furthermore, $\omega(u_0, \phi_0)$ consists of critical points of the energy functional $E$, which is constant on $\omega(u_0, \phi_0)$.
\end{lem}

In the remaining part of this section, we prove Theorem \ref{thm:Eqm}, namely, the set $\omega(u_0, \phi_0)$ is indeed a singleton and moreover, an estimate on the convergence rate can be obtained.

\subsection{Convergence to equilibrium}

By definition of $\omega(u_0, \phi_0)$ and Lemma \ref{omega}, there exists an element $(u_*, \phi_*)\in \omega(u_0, \phi_0)$ and a sequence $t_k\nearrow +\infty$ such that
\begin{align}\label{Eqm:tk}
\lim_{t_k \to +\infty} \norm{(u(t_k), \phi(t_k)) - (u_*, \phi_*)}_{\bW} = 0.
\end{align}
Then we prove
\begin{align}\label{Eqm}
\lim_{t \to +\infty} \norm{(u(t), \phi(t)) - (u_*, \phi_*)}_{\bW} = 0.
\end{align}

\paragraph{Case 1.} Suppose there is a $t_* > 0$ such that $E(u(t_*), \phi(t_*)) = E(u_*,\phi_*):=E_\infty$.  Then, by the non-increasing property of $E(u(t), \phi(t))$ with respect to $t$, it holds that $E(u(t), \phi(t)) = E_\infty$ for all $t \geq t_*$.  In particular, by the energy identity \eqref{E1} it holds that
\begin{align*}
\int_{t_*}^t \norm{\pd_t u(s)}_{L^2(\Omega)}^2 + \norm{\pd_t \phi(s)}_{L^2(\Gamma)}^2 \ds = 0, \quad \forall\, t\geq t_*,
\end{align*}
and so $(u,\phi)$ is independent of time after $t_*$.  Employing this fact together with \eqref{Eqm:tk} leads to the desired convergence \eqref{Eqm}.

\paragraph{Case 2.} Suppose that $E(u(t), \phi(t)) >E_\infty$ for all $t\geq 0$.
For strong solution $(u,\phi)$ to problem \eqref{ACAC} we obtain from \eqref{defn:MM} that
\begin{align*}
(M(u,\phi), (w,\xi))_{\bH} & = \int_\Omega (-\Lap u + f(u)) w \dx + \int_\Gamma (\pdnu u + K^{-1}(u - h(\phi)) ) w \dH \\
& \quad + \int_\Gamma ( - \LB \phi + f_\Gamma(\phi) + K^{-1} h'(\phi) (h(\phi) - u) ) \xi \dH \\
& = \int_\Omega (-\Lap u + f(u)) w \dx + \int_\Gamma ( - \LB \phi + f_\Gamma(\phi) + h'(\phi)\pdnu u ) \xi \dH\\
& = \int_\Omega (\partial_t u) w \dx +\int_\Gamma (\partial_t \phi) \xi \dH.
\end{align*}
Let $\beta$ be the constant in Theorem \ref{thm:LSa} associated with $(u_*, \phi_*)$. For any $t\geq 0$, if the strong solution $(u(t),\phi(t))$ satisfies $\norm{(u(t), \phi(t)) - (u_*, \phi_*)}_{\bV} < \beta$, then we deduce from Theorem \ref{thm:LSa} that
\begin{align}\label{LSc}
C\Big ( \norm{\pd_t u(t)}_{L^2(\Omega)} + \norm{\pd_t \phi(t)}_{L^2(\Gamma)} \Big) \geq |E(u(t),\phi(t))-E_\infty|^{1-\theta}.
\end{align}
Using the above fact and the basic energy law \eqref{Lyap}, one can argue in the exact same manner as in \cite{J} (see also \cite[Section 4.1]{SW}) to conclude that there exists a $t_* > 0$ such that for all $t \geq t_*$, $\norm{(u(t), \phi(t)) - (u_*, \phi_*)}_{\bV} < \beta$, and then \eqref{LSc} holds for all $t\geq t_*$.
As a consequence,
\begin{align*}
-\frac{d}{dt} (E(u(t), \phi(t)) - E_\infty)^{\theta} & = -\theta (E(u(t), \phi(t)) - E_\infty)^{\theta-1} \frac{d}{dt} E(u(t), \phi(t)) \\
& = \theta (E(u(t),\phi(t)) - E_\infty)^{\theta-1} \norm{(\pd_t u(t), \pd_t \phi(t))}_{\bH}^2 \\
& \geq  C\theta \Big ( \norm{\pd_t u(t)}_{L^2(\Omega)} + \norm{\pd_t \phi(t)}_{L^2(\Gamma)} \Big)^{-1} \norm{(\pd_t u(t), \pd_t \phi(t))}_{\bH}^2\\
& \geq  C \Big ( \norm{\pd_t u(t)}_{L^2(\Omega)} + \norm{\pd_t \phi(t)}_{L^2(\Gamma)} \Big),
\end{align*}
where we have used the easy fact $\sqrt{a^2 + b^2} \geq \frac{1}{\sqrt{2}} (a+b)$. Therefore,
\begin{align}\label{Eqm:ODE}
\frac{d}{dt} (E(u(t), \phi(t)) - E_\infty)^{\theta} + C_*\Big ( \norm{\pd_t u(t)}_{L^2(\Omega)} + \norm{\pd_t \phi(t)}_{L^2(\Gamma)} \Big) \leq 0,
\end{align}
for $t\geq t_*$.
Integrating with respect to time yields that
\begin{align*}
\int_{t_*}^{+\infty} \Big ( \norm{\pd_t u(t)}_{L^2(\Omega)} + \norm{\pd_t \phi(t)}_{L^2(\Gamma)} \Big) \dt  < +\infty,
\end{align*}
which together with \eqref{Eqm:tk} implies that $(u(t),\phi(t))$ converges to $(u_*, \phi_*)$ in $\bH$ as $t \to +\infty$.
Thanks to the fact that $(u,\phi) \in L^{\infty}(\delta, +\infty; H^3(\Omega) \times H^3(\Gamma))$, then by compactness we can deduce the convergence \eqref{Eqm}.

\subsection{Convergence rates}
In the second step, we derive estimates on the rate of convergence.

\paragraph{Estimates in $\bH$.} The above analysis asserts that for $t \geq t_*$ the solution $(u(t), \phi(t))$ enters a neighbourhood of a particular equilibrium $(u_*, \phi_*)$ and remains there, namely \eqref{LSc} holds. Hence, from \eqref{Lyap} we obtain for $t \geq t_*$
\begin{align*}
& \frac{d}{dt} (E(u(t), \phi(t)) - E_\infty) + C (E(u(t), \phi(t)) - E_\infty)^{2(1-\theta)} \leq  0,
\end{align*}
which yields (see \cite[Lemma 2.6]{HJ01})
\begin{equation}\label{rate1}
\begin{aligned}
E(u(t), \phi(t)) - E_\infty \leq C (1+t)^{\frac{-1}{1-2 \theta}}\quad \forall\, t\geq t_*.
\end{aligned}
\end{equation}
Then, we infer from \eqref{Eqm:ODE} that
\begin{align*}
& \norm{u(t)-u_*}_{L^2(\Omega)}+\norm{\phi(t)-\phi_*}_{L^2(\Gamma)}\\
 &\quad \leq \int_t^{\infty} \Big ( \norm{\pd_t u(s)}_{L^2(\Omega)} + \norm{\pd_t \phi(s)}_{L^2(\Gamma)} \Big) \ds\\
 &\quad \leq C (1+t)^{\frac{-\theta}{1-2 \theta}}\quad  \forall\, t \geq t_*.
\end{align*}
Together with the uniform estimate of $\|(u, \phi)\|_{\bV}$, we can conclude
\begin{align}\label{H:rate}
\norm{(u(t), \phi(t)) - (u_*, \phi_*)}_{\bH} \leq C (1+t)^{\frac{-\theta}{1-2 \theta}}\quad \forall\, t \geq 0.
\end{align}

\paragraph{Estimates in $\bV$.} The higher-order estimate turns out to be more involved. Let $\hat u := u - u_*$, $\hat \phi := \phi - \phi_*$ with
\begin{align*}
\hat f := f(u) - f(u_*), \quad \hat f_\Gamma := f_\Gamma(\phi) - f_\Gamma(\phi_*), \quad \hat h := h(\phi) - h(\phi_*), \quad \hat{h}' := h'(\phi) - h'(\phi_*).
\end{align*}
Then, subtracting the stationary problem \eqref{stat} from the evolution equations \eqref{ACAC} yields
\begin{subequations}\label{sys:rate}
\begin{alignat}{2}
\pd_t  \hat u - \Lap \hat u + \hat f = 0 &\quad  \text{ a.e. in } \Omega, \label{r:u} \\
K \pdnu \hat u + \hat u = \hat h &\quad  \text{ a.e. on } \Gamma, \label{r:pdnu}  \\
\pd_t \hat \phi - \LB \hat \phi + \hat f_\Gamma + \hat{h}' \pdnu u + h'(\phi_*) \pdnu \hat u = 0 &\quad  \text{ a.e. on } \Gamma. \label{r:phi}
\end{alignat}
\end{subequations}
Testing \eqref{r:u} with $\hat u$ and \eqref{r:phi} with $\hat \phi$ leads to after summing
\begin{equation}\label{Rate1}
\begin{aligned}
& \frac{1}{2} \frac{d}{dt} \norm{(\hat u, \hat \phi)}_{\bH}^2 + \norm{(\nabla \hat u, \surf \hat \phi)}_{\bH}^2 + K^{-1} \norm{\hat u}_{L^2(\Gamma)}^2 \\
& \quad = -\int_\Omega \hat f \hat u \dx - \int_\Gamma \hat f_\Gamma \hat \phi + K^{-1} \Big (h'(\phi_*) (\hat h - \hat u) \hat \phi + (h(\phi) - u) \hat{h}' \hat \phi -  \hat h \hat u \Big ) \dH.
\end{aligned}
\end{equation}
From the fact $\pd_t \hat u = \pd_t u$ we see
\begin{align*}
\int_\Omega (f(u) - f(u_*))\pd_t \hat u \dx =  \frac{d}{dt} \int_\Omega F(u) - F(u_*) - f(u_*)(u - u_*) \dx,
\end{align*}
and so, testing \eqref{r:u} with $\pd_t \hat u$ and \eqref{r:phi} with $\pd_t \hat \phi$ leads to after summing
\begin{equation}\label{Rate2}
\begin{aligned}
& \frac{d}{dt} \Big ( \int_\Omega F(u) - F(u_*) - f(u_*) \hat u \dx + \int_\Gamma F_\Gamma(\phi) - F_\Gamma(\phi_*) - f(\phi_*) \hat \phi \dH \Big ) \\
&\qquad + \frac{1}{2}  \frac{d}{dt} \Big (\norm{(\nabla \hat u, \surf \hat \phi)}_{\bH}^2  + K^{-1} \norm{\hat u}_{L^2(\Gamma)}^2 \Big ) + \norm{(\pd_t \hat u, \pd_t \hat \phi)}_{\bH}^2 \\
& \quad = \int_\Gamma K^{-1} \Big ( \hat h \pd_t \hat u - h'(\phi_*) (\hat h - \hat u) \pd_t \hat \phi - (h(\phi) - u) \hat{h}' \pd_t \hat \phi \Big ) \dH \\
& \quad =  \int_\Gamma K^{-1} \Big ( (\hat h  - h'(\phi_*) \hat \phi) \pd_t \hat u  - (h'(\phi_*) \hat h + (h(\phi) - u) \hat{h}') \pd_t \hat \phi \Big ) \dH\\
& \qquad + \frac{1}{K} \frac{d}{dt} \int_\Gamma h'(\phi_*) \hat u \hat \phi \dH.
\end{aligned}
\end{equation}
Let
\begin{equation}\label{defn:y}
\begin{aligned}
y(t) & := \frac{1}{2} \norm{(\hat u(t), \hat \phi(t))}_{\bV}^2 + \frac{1}{2K} \norm{\hat u(t)}_{L^2(\Gamma)}^2  + \int_\Omega F(u(t)) - F(u_*) - f(u_*)\hat u(t) \dx \\
& \quad + \int_\Gamma F_\Gamma(\phi(t)) - F_\Gamma(\phi_*) - f_\Gamma(\phi_*) \hat \phi(t) - \frac{1}{K} h'(\phi_*) \hat u(t) \hat \phi(t)\dH,
\end{aligned}
\end{equation}
so that upon adding \eqref{Rate1} and \eqref{Rate2} we have
\begin{equation}\label{Rate3}
\begin{aligned}
& \frac{d}{dt} y(t) + \norm{(\pd_t \hat u, \pd_t \hat \phi)}_{\bH}^2 + \norm{(\nabla \hat u, \surf \hat \phi)}_{\bH}^2 + K^{-1} \norm{\hat u}_{L^2(\Gamma)}^2 \\
& \quad = -\int_\Omega \hat f \hat u \dx - \int_\Gamma \hat f_\Gamma \hat \phi + K^{-1} \Big (h'(\phi_*) (\hat h - \hat u) \hat \phi + (h(\phi) - u) \hat{h}' \hat \phi -  \hat h \hat u \Big ) \dH \\
& \qquad + \int_\Gamma K^{-1}\Big ( (\hat h  - h'(\phi_*) \hat \phi) \pd_t \hat u  - (h'(\phi_*) \hat h + (h(\phi) - u) \hat{h}') \pd_t \hat \phi \Big )  \dH.
\end{aligned}
\end{equation}
Since $u\in L^\infty(0,+\infty; H^2(\Omega))$, $u_*\in H^2(\Omega)$, then, by the fundamental theorem of calculus we infer
\begin{equation}\label{f:diff}
\begin{aligned}
\norm{f(u) - f(u_*)}_{L^2(\Omega)}^2 & = \int_\Omega \abs{\int_0^1 f'(u + s(u_* - u)) (u_* - u) \ds}^2 \dx\\
&  \leq C \norm{u-u_*}_{L^2(\Omega)}^2,
\end{aligned}
\end{equation}
and in a similar fashion,
\begin{align*}
\norm{f_\Gamma(\phi) - f_\Gamma(\phi_*)}_{L^2(\Gamma)} \leq C \norm{\phi-\phi_*}_{L^2(\Gamma)}.
\end{align*}
Then we have
\begin{align*}
\abs{\int_\Omega \hat f \hat u \dx} \leq C \norm{\hat u}_{L^2(\Omega)}^2, \quad \abs{\int_\Gamma \hat f_\Gamma \hat \phi \dH } \leq C \norm{\hat \phi}_{L^2(\Gamma)}^2.
\end{align*}
  Moreover, by \eqref{ass:h} and the fact that $h(\phi) - u \in L^{\infty}(0,+\infty;L^{\infty}(\Gamma))$, we obtain
\begin{align*}
K^{-1}\abs{\int_\Gamma h'(\phi_*) (\hat h - \hat u) \hat \phi + (h(\phi) - u) \hat{h}' \hat \phi - \hat h \hat u \dH} \leq \frac{1}{2K} \norm{\hat u}_{L^2(\Gamma)}^2 + C \norm{\hat \phi}_{L^2(\Gamma)}^2
\end{align*}
and for some $\zeta > 0$
\begin{align*}
& K^{-1} \abs{\int_\Gamma (\hat h  - h'(\phi_*) \hat \phi) \pd_t \hat u  - (h'(\phi_*) \hat h + (h(\phi) - u) \hat{h}') \pd_t \hat \phi \dH } \\
& \quad \leq \zeta \norm{\pd_t \hat u}_{L^2(\Gamma)}^2 + \frac{1}{2} \norm{\pd_t \hat \phi}_{L^2(\Gamma)}^2 + C \norm{\hat \phi}_{L^2(\Gamma)}^2.
\end{align*}
Substituting these into \eqref{Rate3} yields
\begin{equation}\label{Rate4}
\begin{aligned}
& \frac{d}{dt} y(t) + \frac{1}{2} \norm{(\pd_t \hat u, \pd_t \hat \phi)}_{\bH}^2 + \norm{(\nabla \hat u, \surf \hat \phi)}_{\bH}^2 + \frac{1}{2K} \norm{\hat u}_{L^2(\Gamma)}^2 \\
& \quad \leq C \norm{(\hat u, \hat \phi)}_{\bH}^2 + \zeta \norm{\pd_t \hat u}_{L^2(\Gamma)}^2.
\end{aligned}
\end{equation}
Recalling the estimate \eqref{time:der}, which reads in our present setting as
\begin{align}\label{Rate5}
\frac{d}{dt} \norm{(\pd_t \hat u, \pd_t \hat \phi)}_{\bH}^2 + \norm{(\nabla \pd_t \hat u, \surf \pd_t \hat \phi)}_{\bH}^2 \leq C_1 \norm{(\pd_t \hat u, \pd_t \hat \phi)}_{\bH}^2,
\end{align}
for some positive constant $C_1$.  Choose a constant $\eta > 0$ such that $\kappa := \frac{1}{2} - C_1 \eta > 0$, then multiplying \eqref{Rate5} with a constant $\eta > 0$ and adding the result to \eqref{Rate4} leads to
\begin{align*}
& \frac{d}{dt} \Big ( \eta \norm{\pd_t \hat u, \pd_t \hat \phi)}_{\bH}^2 + y(t) \Big ) + \min(\kappa, \eta) \norm{(\pd_t \hat u, \pd_t \hat \phi)}_{\bV}^2 + \norm{(\nabla \hat u, \surf \hat \phi)}_{\bH}^2 + \frac{1}{2K} \norm{\hat u}_{L^2(\Gamma)}^2 \\
& \quad \leq C \norm{(\hat u, \hat \phi)}_{\bH}^2 + \zeta \norm{\pd_t \hat u}_{L^2(\Gamma)}^2.
\end{align*}
By choosing $\zeta$ such that $\zeta C_{\mathrm{tr}} < \frac{1}{2} \min(\kappa, \eta)$, where $C_{\mathrm{tr}}$ is the constant from the trace theorem, we can absorb the term $\norm{\pd_t \hat u}_{L^2(\Gamma)}^2$ on the right-hand side by the term $\norm{(\pd_t \hat u, \pd_t \hat \phi)}_{\bV}^2$ on the left-hand side, which leads to
\begin{equation}\label{Rate6}
\begin{aligned}
& \frac{d}{dt} \Big ( \eta \norm{(\pd_t \hat u, \pd_t \hat \phi)}_{\bH}^2 + y(t) \Big ) + c_{\kappa,\eta} \norm{(\pd_t \hat u, \pd_t \hat \phi)}_{\bV}^2 + \norm{(\nabla \hat u, \surf \hat \phi)}_{\bH}^2  + \frac{1}{2K} \norm{\hat u}_{L^2(\Gamma)}^2 \\
& \quad \leq C \norm{(\hat u, \hat \phi)}_{\bH}^2.
\end{aligned}
\end{equation}
Meanwhile, we observe by the Newton--Leibniz formula and the fact $u, u_* \in L^{\infty}(0,+\infty;L^{\infty}(\Omega))$ that
\begin{align*}
\abs{ \int_\Omega F(u) - F(u_*) - f(u_*) \hat u \dx} & = \abs{\int_\Omega \int_0^1 \int_0^1 f'(sz u + (1-sz) u_*) (\hat u)^2 \ds \dz \dx} \\
& \leq C \norm{\hat u}_{L^2(\Omega)}^2,
\end{align*}
and a similar estimate holds for the term involving $F_\Gamma$.   So, from the definition \eqref{defn:y} of $y(t)$ we infer that
\begin{align}\label{y:lb}
y(t) \geq \frac{1}{2} \norm{(\hat u(t), \hat \phi(t))}_{\bV}^2 - C \norm{(\hat u(t), \hat \phi(t))}_{\bH}^2.
\end{align}
On the other hand, we also have
\begin{align}\label{y:up}
\norm{(\nabla \hat u(t), \surf \hat \phi(t))}_{\bH}^2 + \frac{1}{2K} \norm{\hat u(t)}_{L^2(\Gamma)}^2 & \geq y(t) - C \norm{(\hat u(t), \hat \phi(t))}_{\bH}^2.
\end{align}
Substituting  \eqref{H:rate}, \eqref{y:up} into \eqref{Rate6}, there exists a constant $\gamma > 0$ such that
\begin{equation}\label{Rate7}
\begin{aligned}
& \frac{d}{dt} Y(t) + \gamma Y(t)\leq C \norm{(\hat u(t), \hat \phi(t))}_{\bH}^2 \leq C (1+t)^{\frac{-2 \theta}{1-2\theta}},
\end{aligned}
\end{equation}
 where  $Y(t) = \eta \norm{(\pd_t \hat u(t), \pd_t \hat \phi(t))}_{\bH}^2 + y(t)$.  As a result, we can deduce that (cf. \cite{Wu})
\begin{align*}
Y(t) \leq C (1+t)^{\frac{-2 \theta}{1-2\theta}} \quad\forall\, t \geq 0,
\end{align*}
so that by \eqref{y:up} we obtain
\begin{align}\label{V:rate}
\norm{(\hat u(t), \hat \phi(t))}_{\bV}^2 + \norm{(\pd_t \hat u(t), \pd_t \hat \phi(t))}_{\bH}^2 \leq C (1+t)^{\frac{-2 \theta}{1-2\theta}} \quad\forall\, t \geq 0.
\end{align}

\paragraph{Estimates in $\bW$.} To deduce the convergence rate in the $\bW$-norm, we apply elliptic regularity estimates for the system \eqref{sys:rate}, whilst employing \eqref{ass:h}, \eqref{f:diff} and \eqref{V:rate} to obtain
\begin{align*}
&\norm{\hat \phi(t)}_{H^2(\Gamma)} \\
&\quad \leq C \norm{(\pd_t \hat \phi + \hat f_\Gamma + K^{-1} (\hat{h}' (h(\phi)- u) + h'(\phi_*) (\hat h - \hat u)))(t))}_{L^2(\Gamma)} + C \norm{\hat \phi(t)}_{H^1(\Gamma)}, \\
&\quad  \leq C (1+t)^{\frac{-\theta}{1-2 \theta}} \quad\forall\, t \geq 0, \\
&\norm{\hat u(t)}_{H^2(\Omega)}\\
&\quad \leq C \norm{\pd_t \hat u(t) + \hat f(t)}_{L^2(\Omega)} + C\norm{\hat u(t)}_{H^1(\Omega)} + C\norm{K^{-1}(\hat h(t) - \hat u(t))}_{H^{\frac{1}{2}}(\Gamma)}\\
&\quad \leq C (1+t)^{\frac{- \theta}{1-2 \theta}} + C\norm{\hat h(t)}_{H^1(\Gamma)} + C\norm{\hat u(t)}_{H^1(\Omega)} \\
&\quad \leq C (1+t)^{\frac{-\theta}{1-2\theta}}  \quad\forall\, t \geq 0.
\end{align*}
The proof of Theorem \ref{thm:Eqm} is complete.

\section*{Acknowledgements}
K.F. Lam expresses his gratitude to School of Mathematical Sciences at Fudan University for the hospitality during his visit in which part of this research was completed, and gratefully acknowledges the support from a Direct Grant of CUHK (project 4053288). H. Wu is partially supported by NNSFC grant No.~11631011 and the Shanghai Center for Mathematical Sciences at Fudan University.

\end{document}